\documentclass[12pt,a4paper]{article}

\usepackage{graphicx}
\usepackage{amssymb}
\usepackage{amsmath}
\usepackage{epstopdf}

\usepackage{float}

\usepackage{caption}
\usepackage{subcaption}

\DeclareMathOperator*{\cA}{\mathcal A}
\DeclareMathOperator*{\cB}{\mathcal B}

\usepackage[margin=0.8in]{geometry} % for the margins (for a 12pt document, the default margin is 1.5 in)

\usepackage{subfloat}

\title{A Transform Pair for Doubly Connected Domains}

\author{Jesse J. Hulse}

\date{}

\graphicspath{{Fig/}}

\numberwithin{equation}{section}

\usepackage{pst-plot,booktabs,mathtools}
\usepackage{amsfonts, amsthm, xcolor}
\usepackage{graphicx}

\newcommand{\dd}{\mathrm{d}}
\newcommand{\ee}{\mathrm{e}}
\newcommand{\ci}{\mathrm{i}}

\DeclareMathOperator*{\Imag}{Im}
\DeclareMathOperator*{\Real}{Re}
\DeclareMathOperator*{\C}{\mathbb C}
\DeclareMathOperator*{\D}{\mathbb D}
\DeclareMathOperator*{\R}{\mathbb R}
\DeclareMathOperator*{\A}{\mathbb A}

\newtheorem{theorem}{Theorem}[section]

\newtheorem{example}[theorem]{Example}

\theoremstyle{definition}
\newtheorem{definition}[theorem]{Definition}

\theoremstyle{remark}

\usepackage{color}

\usepackage{subcaption}

\begin{document}

\maketitle

%\tableofcontents

\begin{center}
Department of Mathematics \\
University of Manitoba \\
Winnipeg, MB, R3T 2N2, CA
\end{center}

\vskip 0.5truein
\begin{center}
{\bf Subject Areas}
\end{center}
\begin{center}
  complex analysis, unified transform method, boundary value problems
\end{center}

\begin{center}
{\bf Keywords}
\end{center}
\begin{center}
 transform pairs, Szeg\H o kernel, doubly connected domains, applications of complex analysis
\end{center}

\begin{center}
{\bf Abstract}
\end{center}

\noindent A new transform-based technique that generalizes the unified transform method is developed for bounded doubly connected domains as a novel way to numerically solve boundary value problems for holomorphic functions and solutions to the Laplacian. This work builds on the transform methods for multiply connected circular domains developed by Crowdy (2015, IMA J., 80, 1902–1931) and the methods for simply connected bounded domains developed by Hulse et al. (2025, Proc. A, 481 (2319)). The Szeg\"{o} kernel of the annulus and a corresponding transformation law is pivotal in the derivation of this new technique. The modified Schwarz problem and a boundary value problem for a point vortex inside of a doubly connected domain are implemented in detail to demonstrate the effectiveness of this new method.

\vspace{1cm}
\noindent
%\textit{Keywords: }transform pair; harmonic function; mixed boundary value problem.

\vfill\eject

\section{Preliminaries}
\subsection{Finitely connected domains}\label{Prelim}
Let $D$ be a finitely connected bounded domain in $\C$ with $C^\infty$ smooth boundary. We adopt the following notation as in \cite{Bell:1992}. Let $\partial D$ denote the boundary of $D$, which consists of (say) $n$ smooth simple closed curves. We choose $C^\infty$ complex valued functions $\zeta_j(t)$, $j=1,\dots,n$, $t\in [a_j,b_j]$, that parametrize the $n$ boundary curves of $\partial \Omega$ in the standard way. This means that $\zeta_j$ and all of its derivative agrees at $t=a_j$ and $t=b_j$, $\zeta'_j(t)$ is nowhere vanishing, $\zeta_j(t)$ traces out the curve exactly once, and 
\begin{equation*}-i\zeta'_j(t)
\end{equation*}
represents the direction of the outward pointing normal vector to the boundary at the point $\zeta_j(t)$. The unit tangent function $T_{\partial D}(\zeta)$, for $\zeta\in \partial D$, represents the complex number of modulus 1 that gives the direction of the tangent vector to $\partial D$ at $\zeta$ pointing in the direction of the standard orientation of the boundary. The unit tangent vector $T_{\partial D}$ is a continuously differential function that can be written explicitly as $T_{\partial D}(\zeta_j(t))=\zeta'_j(t)/|\zeta'_j(t)|$. The differential $d\zeta$ is given by $d\zeta=\zeta'(t)dt$ and the arclenth differential, denoted by $d\sigma$, is given by $d\sigma(\zeta)=|\zeta'(t)|dt$. It follows that
\begin{equation*}
    d\zeta=T_{\partial D}(\zeta)d\sigma(\zeta),
\end{equation*}
or equivalently
\begin{equation}\label{E:dz-ds}
    d\sigma(\zeta)=\overline{T_{\partial D}(\zeta)}d\zeta,
\end{equation}
for all points $\zeta\in\partial D$.

% Let $D\subset \C$ be a bounded domain in $\mathbb C$. If the boundary of $D$, denoted by $\partial D$, is of class $C^1$, the unit tangent vector,
% \begin{align}
%     T_D(\zeta)=\frac{\zeta'(t)}{|\zeta'(t)|},
% \end{align}
% is a continuously differentiable function of $\zeta \in \partial \D$. One has $d\zeta=\zeta'(t)dt$ and $d\sigma(\zeta)=|z'(t)|dt$ (\cite{Bell:1992}, p. 3), thus it follows that 
% \begin{equation}\label{E:dz-ds}
% d\sigma(\zeta)=\overline{T_{\mathbb D}(\zeta)}d\zeta.
% \end{equation}

% For less regular domains, such as Lipschitz domains, the unit tangent vector exists for $\sigma$-almost every $\zeta\in \partial D$. 
Let $C^\infty(\partial D)$ denote the $C^\infty$ smooth functions on $\partial D$. For $u$ and $v$ in 
 $C^\infty(\partial D)$, the $L^2$ inner product on $\partial 
 D$ of $u$ and $v$ is defined by 
 \begin{equation}\label{innerProd}\langle u,v\rangle_{L^2(\partial D, \dd \sigma)}=\int_{\partial D}u\overline v \hspace{1mm}d\sigma.\end{equation} The space 
 $L^2(\partial D)$ is defined to be the Hilbert space obtained by 
 completing the space $C^\infty(\partial D)$ with respect to this 
 inner product. One may see that $L^2(\partial D)$ is equal to the 
 set of complex valued functions $u$ on $\partial D$ such that 
 $u(\zeta_j(t))$ is a measurable function of $t$ for each $j$, 
 \begin{equation*} ||u||^2=\int_{\partial D}|u|^2d\sigma =\sum_{j=1}^n\int_{a_j}^{b_j}|u(\zeta_j(t))|^2|\zeta'_j(t)|dt
 \end{equation*}
 is finite, and that this 
 definition is independent of the choice of the parametrization of the 
 boundary \cite[p. 11]{Bell:1992}.

 % \textcolor{red}{Note that we have the following formula for the arclength differential $d\sigma$ in terms of the complex differential $d\zeta$ 
 % \begin{align}\label{arclegnthFormula}
 %     d\sigma(\zeta)=\overline{T_{\partial D}(\zeta)}d\zeta.
 % \end{align}
 % }
 % \textcolor{red}{
 % Note that $d\zeta$ and $T_{\partial D}$ will depend on the orientation of the parametrization $\zeta(t)$, but $d\sigma(\zeta)=|\zeta(t)|dt$ does not. For example, if $D=\{z\in \mathbb C:r<|z|<1\}$, then on $\partial \D=\{ \zeta(t)=e^{2\pi it}:0\leq t\leq 1\}$ we have}
 % \begin{align*}
 %     \textcolor{red}{d\sigma(\zeta)\Big|_{|\zeta|=1}=\overline{T_{\partial \mathbb D}(\zeta)}d\zeta=\overline{i\zeta '(t)}\zeta'(t)dt=\overline{ie^{2\pi it}}ie^{2\pi it}dt=-i^2e^{-2\pi it}e^{2\pi it}dt=dt=|\zeta'(t)|dt,}
 % \end{align*}
 % \textcolor{red}{ and on $\{ \zeta(t)=re^{2\pi it}:0\leq t\leq 1\}$, we have}
 % \begin{align*}
 %      \textcolor{red}{d\sigma(\zeta)\Big|_{|\zeta|=r}=\overline{T_{\partial \mathbb D}(\zeta)}d\zeta=\overline{\frac{\zeta' (t)}{|\zeta'(t)|}}\zeta'(t)dt=\frac{\overline{ire^{2\pi it}}}{r}ire^{2\pi it}dt=-i^2e^{-2\pi it}re^{2\pi it}dt=rdt=|\zeta'(t)|dt,}
 % \end{align*}

For an arbitrary finitely connected domain $D$ in $\C$, we define the \textbf{Hardy space} $H^p(D)$ as the analytic functions $f$ on $D$ such that $|f|^p$ has a harmonic majorant in $D$ \cite{PD70}. The function $f$ belongs to the \textbf{Smirnov space} $E^p(D)$ if there exists a sequence of domains $\{ \Omega_n\}$ whose boundaries $\{C_n\}$ consist of a finite number of rectifiable Jordan curves, such that any $\Omega_n$ eventually contains each compact subset of $D$,  
\begin{equation}
\limsup_{n} \int_{C_n}|f(z)|^p \, d\sigma(z)\leq M<\infty,
\label{Hp condition}
\end{equation}
and the lengths of $C_n$ are bounded \cite{PD70}. If we further assume that $\partial D$ consists of rectifiable Jordan curves, then the condition on the boundedness of the lengths of $C_n$ is superfluous. For the remainder of the paper, we will assume that the domain is finitely connected and that our boundary curves are rectifiable. For  all $f\in E^1(D)$, the classical Cauchy integral formula,
\begin{align*}
    f(z)=\frac{1}{2\pi i}\int_{\partial D}\frac{f(\zeta)}{\zeta-z}d\zeta,
\end{align*}
and Cauchy's theorem,
\begin{align*}
    \int_{\partial D} f(z)dz=0,
\end{align*}
hold for the domain $D$ \cite{PD70}. The space $E^2(D)$ has a reproducing kernel inherited from the Hilbert space structure. The \textbf{Szeg\H o kernel} is defined as the unique Hilbert space reproducing kernel for the space $E^2(D)$ of holomorphic functions. The existence of the Szeg\H o kernel follows from the Riesz representation theorem \cite{Conway:1990}, and we note that the Szeg\H o kernel is domain dependent. The key reproducing property of the Szeg\H o kernel is summarized in the following theorem.

\begin{figure*}[t!]
    \centering
\includegraphics[width=0.4\textwidth]{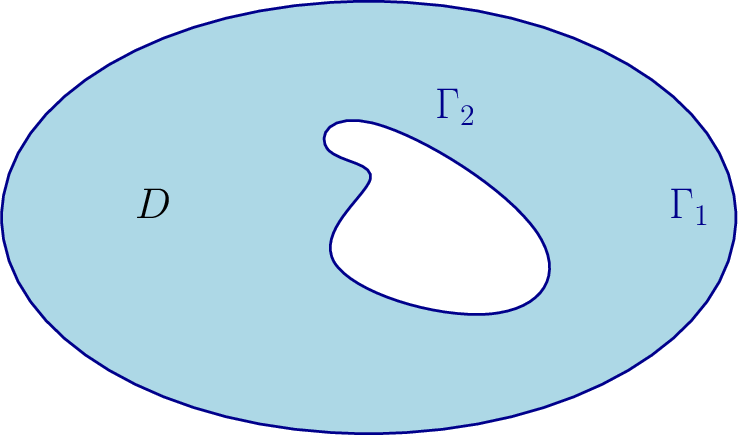}
        
        \label{geometric construction}

    \caption{A doubly connected domain $D$ with outer boundary curve $\Gamma_1$ and inner boundary curve $\Gamma_2$.}
\end{figure*}

\begin{theorem}{(The Szeg\H o formula)} Let $D$ be a bounded finitely connected domain whose boundary consists of $C^\infty$ simple smooth closed curves. For any $f$ in the Smirnov space $E^2(D)$, we have
\begin{equation}\label{E:SF}
f(z)=\int_{\partial D} {f(\zeta) \overline{S_{D}(\zeta, z)} \,\dd \sigma(\zeta)} = \langle f, S_{D}(\cdot, z) \rangle_{L^2(\partial D, \dd \sigma)}, \quad z \in D,
\end{equation}
where $S_{D}(\zeta, z)$ is the  Szeg\H o kernel for $D$ and $d\sigma$ is arclength \cite{Bell:1992}. 
\end{theorem}

\begin{example}
 \noindent The Szeg\H o kernel and Cauchy kernel for the unit disk, denoted by $C_{\D}$, is given by
\begin{equation}\label{E:Szego-disc}
 S_{\mathbb{D}}(\zeta, z)=C_{\mathbb{D}}(\zeta, z)=-\frac{\overline{T_\mathbb{D}(\zeta)}}{2\pi\ci(\overline{\zeta}-\overline{z})} = \frac{1}{2\pi}\frac{1}{(1-\zeta \overline{z})},\quad |\zeta|=1, \quad |z|<1,
 \end{equation}
 where we have used that $T_\mathbb{D}(\zeta)=\ci \zeta$. We note that the unit disk is the only domain on which the Cauchy kernel and Szeg\H o kernel are equal. 
 \end{example}

\begin{example}
The Szeg\H o kernel for the annulus $\A_r :=\{ z\in \C: r<|z|<1\}$ is
\begin{equation}\label{sezgoKernelAnnulus}
   \overline{ S_{\A_r}(z,\zeta)}=\sum_{n=-\infty}^{\infty} \frac{(z\overline\zeta)^n}{1+r^{2n+1}}.
\end{equation}
Any $f$ in the Smirnov space $E^2(\mathbb A_r)$ can be represented by 
\begin{equation}\label{szegoFormula}
     f(z)=\langle f,S_{\A_r}(z,\cdot)\rangle_{E^2}=\int_{\partial \A_r}f(\zeta)\overline{S_{\A_r}(z,\zeta)}d\sigma(\zeta).
\end{equation}

\end{example}

\subsection{Background and connection to the unified transform method}
The Unified Transform Method (UTM) is a technique to analyze and solve boundary value problems for linear and nonlinear integrable PDEs. A.S. Fokas pioneered the UTM \cite{FO97}, and the method has been generalized by many researchers since. The UTM has been studied and formulated for the Laplace, Helmholtz, modified Helmholtz, and biharmonic equations for convex polygons in the plane \cite{DB14,DF15,CH21,HLLL24,SF10}. For the Laplacian, Fokas and Kapaev developed a transform method for simply connected polygonal domains \cite{FK03}. Crowdy developed new Fourier-Mellin transform pairs for Lacplace's equation in simply and multiply connected circular domains. \cite{C15,C15B}. Crowdy and Luca \cite{CL18} later extended the UTM for the biharmonic equation to multiply connected circular domains. 

The present work extends the UTM for harmonic and holomorphic functions to bounded doubly connected domains with smooth boundary curves by utilizing the transformation law for the Szeg\H o kernel. Using the spectral decomposition of the Cauchy kernel developed by Crowdy \cite{C15} and the transformation law for the Szeg\H o kernel, Hulse et al. \cite{HLLS25} extended the UTM for holomorphic functions to bounded, simply connected, Lipschitz domains.

\begin{figure*}[t!]
    \centering
\includegraphics[width=0.3\textwidth]{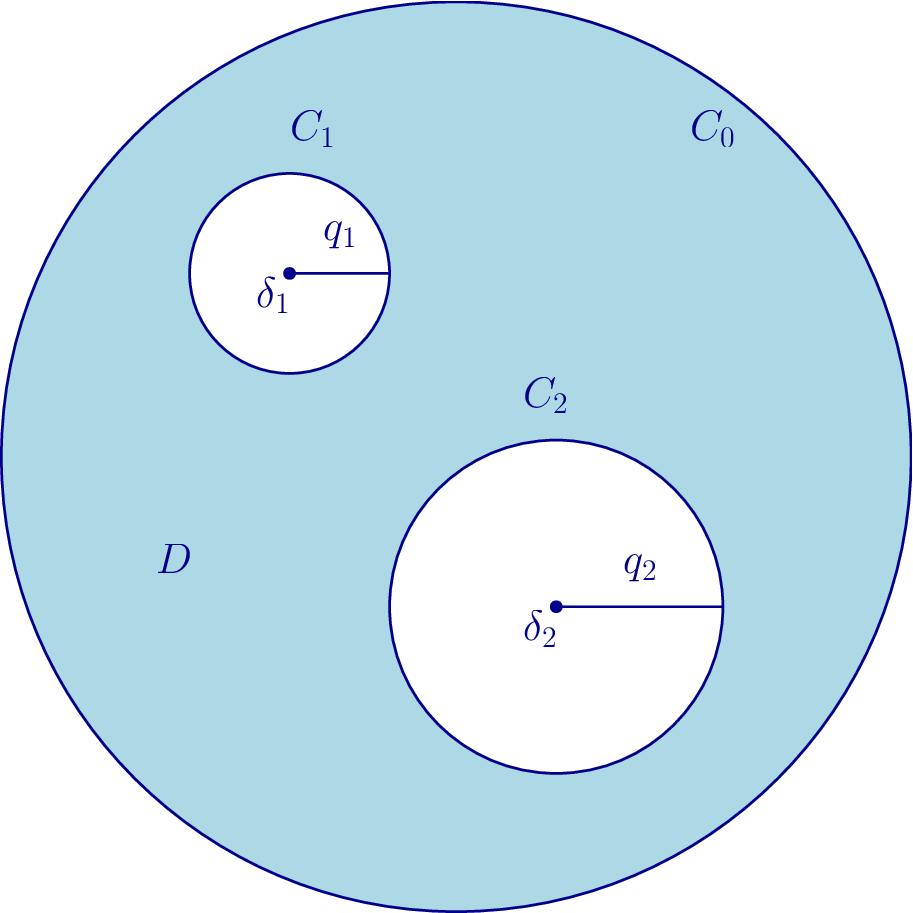}

    \caption{A multiply connected domain $D$ with circular boundary curves.}\label{CircularMultConn}
\end{figure*}

\par Crowdy \cite{C15} uses the Cauchy integral formula to develop a transform pair that extends the UTM to complex valued holomorphic functions on circular multiply connected domains. He does so by finding the following spectral decomposition for the Cauchy kernel for the unit disk:

\begin{equation}\label{E:spectral-disc}
\frac{1}{\zeta-z}=\int_{L_{1}} {\frac{1}{1-\ee^{2\pi \ci k}}\, \frac{z^{k}}{\zeta^{k+1}} \,\dd k}+\int_{L_{2}} {\frac{z^{k}}{\zeta^{k+1}} \,\dd k}+\int_{L_{3}} {\frac{\ee^{2\pi \ci k}}{1-\ee^{2\pi \ci k}}\, \frac{z^{k}}{\zeta^{k+1}} \,\dd k},
\end{equation}
where $L_1,L_2$ and $L_3$ are contours in the spectral k-plane pictured in Figure \ref{fundamentalContour} \cite{C15} and $|z|<1$, $|\zeta|=1$. To be precise, with $0<q<1/2$, $L_1$ is the quarter circle from $-iq$ to $-q$ and the portion of the negative $y$-axis below $-iq$, $L_3$ is the quarter circle from $-q$ to $iq$ and the portion of the positive $y-$axis above $iq$, and $L_2$ is the half line on the real axis starting at $-q$ and emanating in the positive direction. We remark that in \cite{C15,C15B}, $0<q<1$, but we choose $0<q<1/2$ so that the contour is at a positive distance away from the poles appearing in equation (\ref{polesonx=-1/2}). Crowdy \cite{C15} formulates this spectral decomposition for the unit disk by writing the Cauchy kernel as a geometric series and showing
\begin{align*}
    \frac{1}{1-z}=\sum_{n=0}^\infty z^n=\int_{L_1}\frac{1}{1-e^{2\pi ik}}z^kdk+\int_{L_2}z^kdk+\int_{L_3}\frac{e^{2\pi ik}}{1-e^{2\pi ik}}z^kdk.
\end{align*}  Let the domain $D$ be the unit disk with $M$ disks of center $\delta_j$ and radius $q_m$ removed as shown in figure \ref{CircularMultConn}. The transform pair for $D$ is then
\begin{align*}
    &f(z)=\frac{1}{2\pi i}\Big[ \int_{L_1}\frac{\rho_{00}(k)}{1-e^{2\pi ik}}z^kdk+\int_{L_2}\rho_{00}z^kdk+\int_{L_3}\frac{\rho_{00}(k)e^{2\pi ik}}{1-e^{2\pi ik}}z^kdk
    \\&- \sum_{j=1}^m \Big( \int_{L_1}\frac{\rho_{jj}(k)}{1-e^{2\pi ik}}\Big(\frac{q_j}{z-\delta_j}\Big)^{k+1}+\int_{L_2}\rho_{jj}(k)\Big(\frac{q_j}{z-\delta_j}\Big)^{k+1}+\int_{L_3}\frac{\rho_{jj}(k)e^{2\pi ik}}{1-e^{2\pi ik}}\Big(\frac{q_j}{z-\delta_j}\Big)^{k+1}\Big)dk\Big]
\end{align*}
where the spectral matrix is defined for $n=0,1,2,\dots,M$
\begin{align*}
    &\rho_{0n}(k)=\int_{C_n}\frac{f(\zeta)}{\zeta^{k+1}}d\zeta
    \\& \rho_{mn}(k)=-\frac{1}{q_m}\int_{C_n}f(\zeta)\Big[\frac{\zeta-\delta_m}{q_m}\Big]^kd\zeta, \hspace{1cm} m=1,2,\dots,M.
\end{align*}
The global relations are
\begin{align*}
    \sum_{n=0}^M\rho_{mn}(k)=0, \hspace{1cm}k\in -\mathbb N, m=0,1,\dots,M.
\end{align*}
We again remark that the domain $D$ appearing in Crowdy's transform pair above has circular boundary arcs.
\begin{figure}
    \centering
        \includegraphics[width=0.5\textwidth]{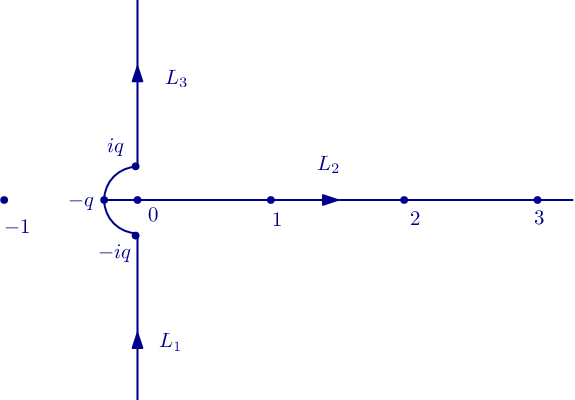} 
        
    \caption{The fundamental contour with $0<q<1/2$.}\label{fundamentalContour}
\end{figure}

\section{Deriving the new spectral decomposition for $S_{\mathbb A_r}$}
The first step in deriving the new transform pair for general domains is finding the spectral decomposition of the Szeg\H o kernel of $\A_r$. Recall the Szeg\H o kernel for the annulus is
\begin{align}
    &\overline{S_{\mathbb A_r}(z,\zeta)}=\sum_{n=-\infty}^\infty\frac{(z\overline \zeta)^{n}}{1+r^{2n+1}}=\sum_{n=-\infty}^{-1}\frac{(z\overline \zeta)^{n}}{1+r^{2n+1}}+\sum_{n=0}^\infty\frac{(z\overline \zeta)^{n}}{1+r^{2n+1}}.
    \end{align}
    Rearranging the first series on the right hand side above gives
    \begin{equation*}
    \sum_{n=-\infty}^{-1}\frac{(z\overline \zeta)^{n}}{1+r^{2n+1}}=\sum_{n=1}^{\infty}\frac{1}{(z\overline \zeta)^{n}}\frac{1}{1+r^{-2n+1}}=\sum_{n=1}^{\infty}\Big(\frac{r^2}{z\overline \zeta}\Big)^n\frac{1}{r+r^{2n}}
    =\frac{r}{z\overline \zeta}\sum_{n=0}^{\infty}\Big(\frac{r^2}{z\overline \zeta}\Big)^n\frac{1}{1+r^{2n+1}}.
\end{equation*}
% \textbf{Remark}: Note that for $|z|<1,|\zeta|=1$, we have
% \begin{align*}  &\lim_{r\rightarrow 0}\frac{1}{2\pi}\Big(\frac{r}{z\overline \zeta}\sum_{n=0}^{\infty}\Big(\frac{r^2}{z\overline \zeta}\Big)^n\frac{1}{1+r^{2n+1}}+\sum_{n=0}^\infty\frac{(z\overline \zeta)^n}{1+r^{2n+1}}\Big)\\&=0+\frac{1}{2\pi}\sum_{n=0}^\infty(z\overline{\zeta})^n=\frac{1}{2\pi}\frac{1}{1-z\overline \zeta}= S_{\mathbb D}(z,\zeta)= C_{\mathbb D}(z,\zeta) \hspace{8mm} 
% \end{align*}
% where $C_{\mathbb D}(z,\zeta)$ is the Cauchy kernel for the disk.

\noindent Recall that Crowdy finds 
\begin{align}\label{CrSpec}
    \frac{1}{1-z}=\sum_{n=0}^\infty z^n=\int_{L_1}\frac{1}{1-e^{2\pi ik}}z^kdk+\int_{L_2}z^kdk+\int_{L_3}\frac{e^{2\pi ik}}{1-e^{2\pi ik}}z^kdk.
\end{align}
It is easily verified that the integrands in (\ref{CrSpec}) decay exponentially as $|k|\rightarrow\infty$ along the corresponding contours for all $|z|<1$. We use a similar argument to show that 

\begin{equation}\label{specdecom1}
\begin{split}
\sum_{n=0}^\infty&\frac{(z\overline{\zeta})^n}{1+r^{2n+1}}\\&=\int_{L_1}\frac{1}{1+r^{2k+1}}\frac{(z\overline{\zeta})^k}{1-e^{2\pi ik}}dk+\int_{L_2}\frac{(z\overline{\zeta})^k}{1+r^{2k+1}}dk+\int_{L_3}\frac{(z\overline{\zeta})^k}{1+r^{2k+1}}\frac{e^{2\pi ik}}{1-e^{2\pi ik}}dk
\end{split}
\end{equation}
and
\begin{equation}\label{specdecoms2}
\begin{split}
\frac{r}{z\overline\zeta}\sum_{n=0}^{\infty}\Big(&\frac{r^2}{z\overline \zeta}\Big)^n\frac{1}{1+r^{2n+1}}=\frac{r}{z\overline\zeta}\Big[\int_{L_1}\frac{1}{1+r^{2k+1}}\frac{1}{1-e^{2\pi ik}}\Big(\frac{r^2}{z\overline \zeta}\Big)^kdk\\&+\int_{L_2}\frac{1}{1+r^{2k+1}}\Big(\frac{r^2}{z\overline \zeta}\Big)^kdk+\int_{L_3}\frac{1}{1+r^{2k+1}}\frac{e^{2\pi ik}}{1-e^{2\pi ik}}\Big(\frac{r^2}{z\overline \zeta}\Big)^kdk\Big].
\end{split}
\end{equation}
See Appendix  \ref{appendixA} for details. Define the function 
\begin{align*}
    \gamma(w,k):=\frac{1}{1+r^{2k+1}}\Big(w^k+\frac{r^{2k+1}}{w^{k+1}}\Big), \hspace{6mm} w\in \overline{\D}\setminus\{0\},\hspace{1mm} k\in \mathbb C.
\end{align*}
Then the Szeg\H o kernel for $\A_r$ has the following spectral decomposition
\begin{equation}\label{spectralSzegoAnnulus}
\begin{split}
    &2\pi S_{\mathbb A_r}(z,\zeta)=\int_{L_1}\frac{\gamma(z\overline{\zeta},k)}{1-e^{2\pi ik}}dk+\int_{L_2}\gamma(z\overline{\zeta},k)dk
    +\int_{L_3}\frac{e^{2\pi ik}}{1-e^{2\pi ik}}\gamma(z\overline{\zeta},k)dk
    \\&=\int_{L_1}\frac{1}{1+r^{2k+1}}\frac{1}{1-e^{2\pi ik}}\big(z\overline{\zeta}\big)^{k+1}dk+\int_{L_2}\frac{1}{1+r^{2k+1}}\big(z\overline{\zeta}\big)^{k+1}dk\\&+\int_{L_3}\frac{1}{1+r^{2k+1}}\frac{e^{2i\pi ik}}{1-e^{2\pi ik}}\big(z\overline{\zeta}\big)^{k+1}dk
    \\&+\int_{L_1}\frac{1}{1+r^{2k+1}}\frac{1}{{1-e^{2\pi ik}}}\Big(\frac{r^2}{z\overline{\zeta}}\Big)^{k+1}dk+\int_{L_2}\frac{1}{1+r^{2k+1}}\Big(\frac{r^2}{z\overline{\zeta}}\Big)^{k+1}dk
    \\&+\int_{L_3}\frac{1}{1+r^{2k+1}}\frac{e^{2i\pi ik}}{{1-e^{2\pi ik}}}\Big(\frac{r^2}{z\overline{\zeta}}\Big)^{k+1}dk.
    \end{split}
\end{equation}
The integrands in (\ref{spectralSzegoAnnulus}) also decay exponentially as $|k|\rightarrow \infty$ along the contours $L_1,L_2$ and $L_3$ for all $z\in \A_r$ and $\zeta\in \partial\A_r $. To see this, first note that $1+r^{2k+1}=0$ for $k\in\mathbb \C$ such that 
\begin{align}\label{polesonx=-1/2}
    k=-\frac{1}{2}+i\frac{(2n+1)\pi}{2\ln(r)}, \hspace{1cm} n\in\mathbb Z.
\end{align}
Thus the poles of $1/(1+r^{2k+1})$ are at least a distance of $1/2$ from the contours $L_1$ and $L_3$ as $|k|\rightarrow \infty$. Further  for $k=ix$ with $x\in \R$, we have
\begin{align*}
    \Big|\frac{1}{1+r^{2k+1}}\Big|=\Big|\frac{1}{1+r^{2ix}r}\Big|\leq\frac{1}{1-r}.
\end{align*}
Thus $1/(1+r^{2k+1})$ is bounded along $L_1$ and $L_3$. It follows that the integrals over $L_1$ and $L_3$ exponentially decay for the same reason they decay in $(\ref{specdecom1})$. Next, let $z=x+iy$ with $x>0$. Then
\begin{align*}
     \Big|\frac{1}{1+r^{2k+1}}\Big|= \Big|\frac{1}{1+r^{2x}r^{2yi}r}\Big|. 
\end{align*}
Given that $x>0$, it follows that quantity above 
 is bounded. Thus the integrands in (\ref{spectralSzegoAnnulus}) decay exponentially as $|k|\rightarrow \infty$ along the corresponding contours.

\subsection{New transform pair for $\mathbb A_r$}
The new spectral decomposition found in (\ref{spectralSzegoAnnulus}) will be used to derive a new transform pair for $\A_r$, but first a definition.
\begin{definition}\label{spectralFcns}
Let $D^1=\D$ and let $D^2=D_r(a)$ denote the disk of radius $r$ centered at $a$. Let the boundary of $ \D$ be parametrized by $\zeta_1(t)=e^{ it}$, $t\in[0,2\pi]$ and the boundary of $D_r(0)$ be parametrized by $\zeta_2(t)=-re^{it}$, $t\in[0,2\pi]$. For $f\in E^2(\mathbb A_r)$ and $z\in \mathbb A_r$, define 
\begin{align*}
        &\rho_{1m}(k)=\int_{\partial  D^m}\frac{f(\zeta_m)}{\zeta_m^{k+1}}d\zeta_m \hspace{1cm} k\in \C,m\in\{1,2\},
        \\& \rho_{2m}(k)=r^{2k+1}\int_{\partial D^m}f(\zeta_m)\zeta_m^{k}d\zeta_m \hspace{1cm} k\in \C,m\in\{1,2\}.
    \end{align*}
    Additionally, define
    \begin{equation}\label{tau}
        \tau(z,k)=\frac{1}{1+r^{2k+1}}\Big[z^k(\rho_{11}(k)+\rho_{22}(-k-1))+\frac{1}{z^{k+1}}(\rho_{21}(k)+\rho_{12}(-k-1))\Big].
    \end{equation}
\end{definition}
\noindent\textbf{Remark 1}: The integrals in definition \ref{spectralFcns} are independent of choice of parametrization, but the parametrization must be oriented in the standard way as described in section \ref{Prelim}. The subscripts in \ref{TransPairAnn} and the subscript appearing in $\zeta_m$ will be dropped for the remainder of this paper, but it is assumed that any parametrization is chosen in this standard way.\\
\noindent\textbf{Remark 2}: 
Crowdy's spectral functions for $\mathbb A_r$ are
\begin{align*}
    &\rho_{0n}=\oint_{C_n}\frac{f(\zeta)}{\zeta^{k+1}}d\zeta
    \\&\rho_{1n}(k)=-\frac{1}{r}\oint_{C_n}f(\zeta)\Big[\frac{\zeta}{r}\Big]^kd\zeta=-\frac{1}{r^{k+1}}\oint_{C_n}f(\zeta)\zeta^kd\zeta,
\end{align*}
for $n\in \{0,1\}$ and $m\in \{1,2\}$. The integrals in Crowdy's spectral $\rho_{2m}$ functions are orientated in the counterclockwise direction, hence the negative sign.

\begin{theorem}\label{TransPairAnn}
    Let $f\in E^2(\mathbb A_r)$ and $z\in\mathbb A_r$. Then $f$ has the following  transform pair representation
    \begin{equation}\label{transformAnnulus}
        f(z)=\frac{1}{2\pi i}\Big[\int_{L_1}\frac{1}{1-e^{2\pi i k}}\tau(z,k)dk+\int_{L_2}\tau(z,k)dk+\int_{L_3}\frac{e^{2\pi  ik}}{1-e^{2\pi  ik}}\tau(z,k)dk\Big],
    \end{equation}
    where $\tau$ is defined by \eqref{tau}. Further, the following global relation holds
    \begin{align}\label{GRA_r}
        \rho_{j1}(k)+\rho_{j2}(k)=0, \hspace{5mm} k\in \mathbb Z,\ j=1,2.
    \end{align}

    \begin{proof}
       By (\ref{szegoFormula}) and (\ref{spectralSzegoAnnulus}), for $f\in E^2(\mathbb A_r)$ and $z\in \mathbb A_r$, we have
\begin{align*}
     &f(z)=\int_{\partial \A_r}f(\zeta) \overline{S_{\mathbb A_r}(z,\zeta)}d\sigma(\zeta)\\&=\frac{1}{2\pi}\int_{\partial\mathbb A_r}f(z)\Big[\int_{L_1}\frac{\tau(z\overline{\zeta},k)}{1-e^{2\pi ik}}dk+\int_{L_2}\tau(z\overline{\zeta},k)dk
    +\int_{L_3}\frac{e^{2\pi ik}}{1-e^{2\pi ik}}\tau(z\overline{\zeta},k)\Big]dkd\sigma(\zeta).
\end{align*}
The modulus of the integrands in the spectral decomposition decays exponentially $|k|\rightarrow\infty$, so we may switch the order of integration. See Appendix B of \cite{HLLS25} for details concerning this type of argument. Using the equation (\ref{E:dz-ds}) and simplifying gives (\ref{transformAnnulus}). See Appendix \ref{appendixB} for these details. Next, for $j=1$, we have
\begin{align*}
    \rho_{11}(k)+\rho_{12}(k)=\int_{\partial \D}\frac{f(\zeta_1)}{\zeta_1^{k+1}}d\zeta_1+\int_{\partial D_r(0)}\frac{f(\zeta_2)}{\zeta_2^{k+1}}d\zeta_2=\int_{\mathbb A_r}\frac{f(\zeta)}{\zeta^{k+1}}d\zeta,
\end{align*}
which is equal to 0 when $k\in\mathbb Z$ by Cauchy's theorem for $E^1(\mathbb A_r)\supseteq E^2(\mathbb A_r)$. Likewise 
\begin{align*}
    \rho_{21}(k)+\rho_{22}(k)=r^{2k+1}\int_{\partial \D}f(\zeta_1)\zeta_1^{k}d\zeta_1+r^{2k+1}\int_{\partial D_r(0)}f(\zeta_2)\zeta_2^{k}d\zeta_2=r^{2k+1}\int_{\mathbb A_r}f(\zeta)\zeta^{k}d\zeta=0.
\end{align*}
Here we have used that fact that $\zeta^k$ is holomorphic on $\A_r$ for all integers $k$. 
    \end{proof}
\end{theorem}
We remark on a key difference between the transform pair developed using the Szeg\H o kernel for the annulus and the transform pair developed using the Szeg\H o kernel for the unit disk in \cite{HLLS25}. For the unit disk, the Szeg\H o kernel and Cauchy kernel are the same, hence the transform pair found in \cite{HLLS25} matches the transform pair found in \cite{C15} when the domain is the unit disk. However, the Szeg\H o kernel for an annulus is not equal to the Cauchy kernel, so the transform pair developed in this work will never equal the transform pair found in \cite{C15B}.

\subsection{New transform pair for doubly connected domains}\label{NewTransPairSec}

One advantage of the Cauchy kernel is that is it domain independent. However, the spectral decomposition (\ref{E:spectral-disc}) only holds on domains with circular boundary curves. While the Szeg\H o kernel is domain dependent, it has a transformation law that allows for the generalization of (\ref{spectralSzegoAnnulus}) to general domains. To be precise, if $\Phi: D_1\mapsto D_2$ is conformal, then
\begin{equation}\label{E:Szegokernel}
S_{D_1}(\zeta,z)\coloneqq \overline{\sqrt{\Phi'(z)}} S_{D_2}(\Phi(\zeta),\Phi(z)) \sqrt{\Phi'(\zeta)}, \quad \zeta \in \partial D_1, \quad z \in D_1.
\end{equation}
If $D_1$ and $D_2$ are smooth bounded domains (possibly finitely connected), then the transformation law above holds. Using the equation in \eqref{E:dz-ds}, we can write
\begin{equation}\label{E:Szego-cont}
\overline{S_{D_1}(\zeta,z)}\,\dd \sigma(\zeta) =\sqrt{\Phi'(z)} ~\overline{S_{D_2}(\Phi(\zeta),\Phi(z))}~\overline{{\sqrt{\Phi'(\zeta)}}}~\overline{T_{D_1}(\zeta)}\, \dd \zeta.
\end{equation}
 One may prove (see \cite[p. 53]{Bell:1992}) that
\begin{equation}\label{E:tg-vecs}
T_{ D_2}(\Phi (\zeta))\overline{\sqrt{\Phi'(\zeta)}}  =
 \sqrt{\Phi'(\zeta)}T_{D_1}(\zeta),\quad \zeta \in \partial D.
\end{equation}
Combining the two equations \eqref{E:tg-vecs} and \eqref{E:Szego-cont} gives
\begin{equation}\label{E:Szego-final}
\overline{S_{D_1}(\zeta,z)}\,\dd \sigma(\zeta) =\sqrt{\Phi'(z)}~\overline{S_{D_2}(\Phi(\zeta),\Phi(z))}~{\sqrt{\Phi'(\zeta)}}~\overline{T_{D_2}(\Phi(\zeta))} \dd \zeta.
\end{equation}
Now let $D_2=\mathbb A_{r}$ and $D_1=D$ where $D$ is a smooth bounded domain such that $\Psi:\mathbb A_r\rightarrow D$ is a conformal map. Substituting in the expression for $S_{\mathbb{A}_r}$ found in \eqref{spectralSzegoAnnulus} into \eqref{E:Szego-final} gives
\begin{equation}
\begin{split}
    &S_{D}(z,\zeta)d\sigma(z)=\sqrt{\Phi'(\zeta)}S_{\mathbb A_r}(\Phi(z),\Phi(\zeta))\sqrt{\Phi'(z)}\hspace{1mm}\overline{iT_{\mathbb A_r}(\zeta)}d\zeta\\&=\frac{\sqrt{\Phi'(z)}}{2\pi i}\Big[\int_{L_1}\frac{\tau(\Phi(z)\Phi(\overline{\zeta}),k)}{1-e^{2\pi ik}}dk+\int_{L_2}\tau(\Phi(z)\Phi(\overline{\zeta}),k)dk
    \\&+\int_{L_3}\frac{e^{2\pi ik}}{1-e^{2\pi ik}}\tau(\Phi(z)\Phi(\overline{\zeta}),k)dk\Big]\sqrt{\Phi'(\zeta)}\overline{T_{\mathbb{A}_r}(\Phi(\zeta))}d\zeta.
    \end{split}
\end{equation}
This leads to the following theorem.

\begin{theorem} Let $D$ be a bounded smooth doubly connected domain conformally equivalent to $\mathbb A_r=\{ z\in \C: r<|z|<1\}$, say $\Phi:D\rightarrow \mathbb A_r$ is a conformal map. Let $\Gamma_1$ be the outer boundary curve of $D$, $\Phi^{-1}(\partial \D)=\Gamma_1$, and let $\Gamma_2$ be the inner boundary of $D$, $\Phi^{-1}(\partial D_r(0))=\Gamma_2$. Then for any $f\in E^2(D)$ and $z\in D$, we have
   \begin{equation}\label{transformDoublyConnected}
        f(z)=\frac{\sqrt{\Phi'(z)}}{2\pi i}\Big[\int_{L_1}\frac{1}{1-e^{2\pi i k}}\tau_\Phi(z,k)dk+\int_{L_2}\tau_\Phi(z,k)dk+\int_{L_3}\frac{e^{2\pi  ik}}{1-e^{2\pi  ik}}\tau_\Phi(z,k)dk\Big],
    \end{equation}
    where
         \begin{equation}\label{tauPhi}
        \tau_\Phi(z,k)=\frac{\Phi(z)^k}{1+r^{2k+1}}(\rho_{11,\Phi}(k)+\rho_{22,\Phi}(-k-1))+\frac{1}{\Phi(z)^{k+1}}(\rho_{21,\Phi}(k)+\rho_{12,\Phi}(-k-1)),
    \end{equation}
    \begin{equation}\label{rhoFcns}
    \begin{split}
        &\rho_{1m,\Phi}(k)=\int_{\Gamma_m}\frac{f(\zeta)\sqrt{\Phi'(\zeta)
        }}{\Phi(\zeta)^{k+1}}d\zeta \hspace{1cm} k\in \C, m=1,2
        \\& \rho_{2m,\Phi}(k)=r^{2k+1}\int_{\Gamma_m}f(\zeta)\Phi(\zeta)^{k}\sqrt{\Phi'(\zeta)
        }d\zeta \hspace{1cm} k\in \C, m=1,2.
        \end{split}
    \end{equation}
    The following global relation also holds
    \begin{align}\label{GR}
      \rho_{j1,\Phi}(k)+\rho_{j2,\Phi}(k)=0, \hspace{5mm} k\in \mathbb N\cup \{0\}, j=1,2.
    \end{align}
    \end{theorem}
\noindent\textbf{Remark 1}: If one considers the transform pair when the inner radius $r$ is zero, then $\rho_{2m,\Phi}(k)=0$ for all $k\in \C$, $\tau_{\Phi}(z,k)=\Phi(k)^k\rho_{1m,\Phi}(z)$, and the transform pair reduces to
\begin{align*}
    f(z)=\frac{\sqrt{\Phi'(z)}}{2\pi i}\Big[\int_{L_1}\frac{\Phi(k)^k\rho_{1,\Phi}(z)}{1-e^{2\pi i k}}dk+\int_{L_2}\Phi(k)^k\rho_{1,\Phi}(z)dk+\int_{L_3}\frac{e^{2\pi  ik}}{1-e^{2\pi  ik}}\Phi(k)^k\rho_{1,\Phi}(z)dk\Big].
\end{align*}
This is exactly the transform pair for simply connected domains found in \cite{HLLS25}.\\
\noindent\textbf{Remark 2:} Note that if $k\in \mathbb Z$ instead of $\mathbb N$, then only a single global relation is needed. However, to emphasize the difference between the simply-connected case in which there is only one global relation and the doubly-connected case, the global relation above will be written as in (\ref{GR}).

\section{The global relation}
We will show that
\begin{align*}
   \rho_{j1,\Phi}(k)+\rho_{j2,\Phi}(k)=0, \hspace{5mm} k\in \mathbb Z, j=1,2.
\end{align*}
We prove the equality above for $j=1$. The proof for $j=2$ follows in a similar manner. This proof is similar to the proof of the corresponding global relation for simply connected domain presented in \cite{HLLS25}. Let $k \in \mathbb{N}\cup\{0\}$ and write $n \coloneqq -k-1$. Write
\begin{equation}
\tilde\rho(n)\coloneqq \rho_{11,\Phi}(-n-1) +\rho_{12,\Phi}(-n-1) = \int_{\partial D} {\sqrt{\Phi'(\zeta)} f(\zeta) \Phi(\zeta)^n \,\dd \zeta}.
\end{equation}
We will prove
\begin{equation}
\tilde\rho(n)=0, \qquad \text{for } n \in \{0\}\cup \mathbb{N},
\label{GR tau(n)}
\end{equation}
and \eqref{GR} then follows.
Denote the inverse of the conformal map $\Phi$ by $\Psi:\mathbb A_r\rightarrow D$. Now, we have the following classical result that can be found in \cite[p. 51]{Bell:1992}. Suppose that $\Psi:D_1\rightarrow D_2$ is a biholormphic mapping between bounded domains with smooth boundaries. Then $\Psi\in C^\infty(\overline{D_1})$, $\Psi'$ is nonvanishing on $\overline{D_1}$, and $\Psi^{-1}\in C^\infty(\overline{D_2})$. Additionally $\Psi'$ is equal to the square of a function that is holomorphic on a neighborhood of $\overline{D_1}$.
Further, it is known \cite[p. 183]{PD70} that
\begin{align}\label{EpHp}
    f\in E^2(\mathbb A_r) \text{ if and only if } f(\Psi(\zeta))[\Psi'(\zeta)]^{1/2}\in H^2( D).
\end{align}
 In particular, $\sqrt{\Psi'(\zeta)}\in H^2(D)$ given that the constant functions are in all Smirnov Spaces. Thus writing $\Psi'=\sqrt{\Psi'}\sqrt{\Psi'}$, applying the change of variable formula $\zeta:=\Psi(w)$, and using the fact that $(\Phi \circ \Psi)'(w)=(\Phi \circ \Phi^{-1})'(w)=1$,  we find

\begin{equation}
\begin{split}
&\tilde\rho(n)=\int_{\partial \mathbb A_r} {\sqrt{\Phi'(\Psi(w))} (f \circ \Psi)(w) w^{n} \Psi'(w) \,\dd w} \\
&=\int_{\partial \mathbb A_r} {\sqrt{\Phi'(\Psi(w)) \Psi'(w)} (f \circ \Psi)(w) [\Phi \circ \Psi]^{n} (w) \sqrt{\Psi'(w)} \,\dd w} \\
&=\int_{\partial \mathbb A_r} {[(f \cdot \Phi^{n}) \circ \Psi](w) \sqrt{\Psi'(w)} \,\dd w}.
\end{split}
\end{equation}
It is known that $E^2(D)=H^2(D)$ if all boundary curves are analytic \cite[p. 182]{PD70}. Thus it follows from (\ref{EpHp}) that $f \cdot \Phi^{n} \in E^2(D)$ if and only if $[f \cdot \Phi^{n}] \circ \Psi \cdot \sqrt{\Psi'} \in H^2(\mathbb{A}_r)=E^2(\mathbb{A}_r)$. It was already noted above that $\Phi$ extends continuously to the boundary. Thus it follows that $f\cdot \Phi^n\in E^2(D)$. The equality in (\ref{GR tau(n)}) now follows from Cauchy's theorem for $E^1$ and the fact that $E^2(D)\subset E^1(D)$ when $D$ is bounded.

\section{Alternative integral representations for $\mathbb A_r$}
In this section, we write alternative integral representations for $S_{\mathbb A_r}$ that arise from the Watson transform. These contours were also noted in \cite{C15} as alternative integral representations for that transform pair. First, we have
\begin{align*}
    \overline{S_{\mathbb A_r}(z,\zeta)}=\int_{C}\frac{(z\overline\zeta)^n}{(1-e^{2\pi ik})r^{2n+1}}dk+\int_{C+1}\frac{1}{(1-e^{2\pi ik)}(r+r^{2n+1})}\Big(\frac{r^2}{z\overline\zeta}\Big)^ndk,
\end{align*}
where $C$ is the left semi-circle of radius $0<\delta<1$ (adequately small) and two horizontal half-lines with endpoints $\pm i\delta$ as pictured in figure (\ref{WatsonContour}). Let $C+1$ be the contour $C$ translated to the right by 1. As was noted in $\cite{C15}$, the contour $C$ is within a distance of 1 from the singularities of the integrand along the real axis. Hence this will not be the optimal choice of contour for computation. The integrands in the spectral decomposition (\ref{spectralSzegoAnnulus}) developed in this work for the Szeg\H o kerenl decay exponentially. However, we may deform the contour $C$ in order to move the contour away from the singularities of the integrands. This new deformed contour $C'$ is pictured in figure (\ref{WatsonContourDeformed}). As before, let $C'+1$ be the contour $C'$ translated to the right by one. Then we have the following integral representation for $S_{\A_r}(z,\zeta)$:
% \noindent
% To this end, we claim that $\Phi$ extends to $\Phi \in C(\overline{D}) \cap \vartheta(D)$. To see this we have that $\Psi$ extends to a homeomorphism $\Psi:\overline{\mathbb{A}_r}\mapsto \overline{D}$ (more precisely, $\Psi \in \vartheta(\mathbb{A}_r) \cap C(\overline{\mathbb{A}_r})$, $\Psi: \partial \mathbb{A}_r \mapsto \partial D$ homeomorphism, $\Psi |_{\mathbb{A}_r}=\psi$). Hence $\Psi^{-1}: \overline{D}\mapsto\overline{\mathbb A}_r $ is a homeomorphism and we may take $\Phi \coloneqq \Psi^{-1}$. 

\begin{figure*}[t!]
    \centering
    \begin{subfigure}[t]{0.5\textwidth}
        \centering
        \includegraphics[width=0.7\textwidth]{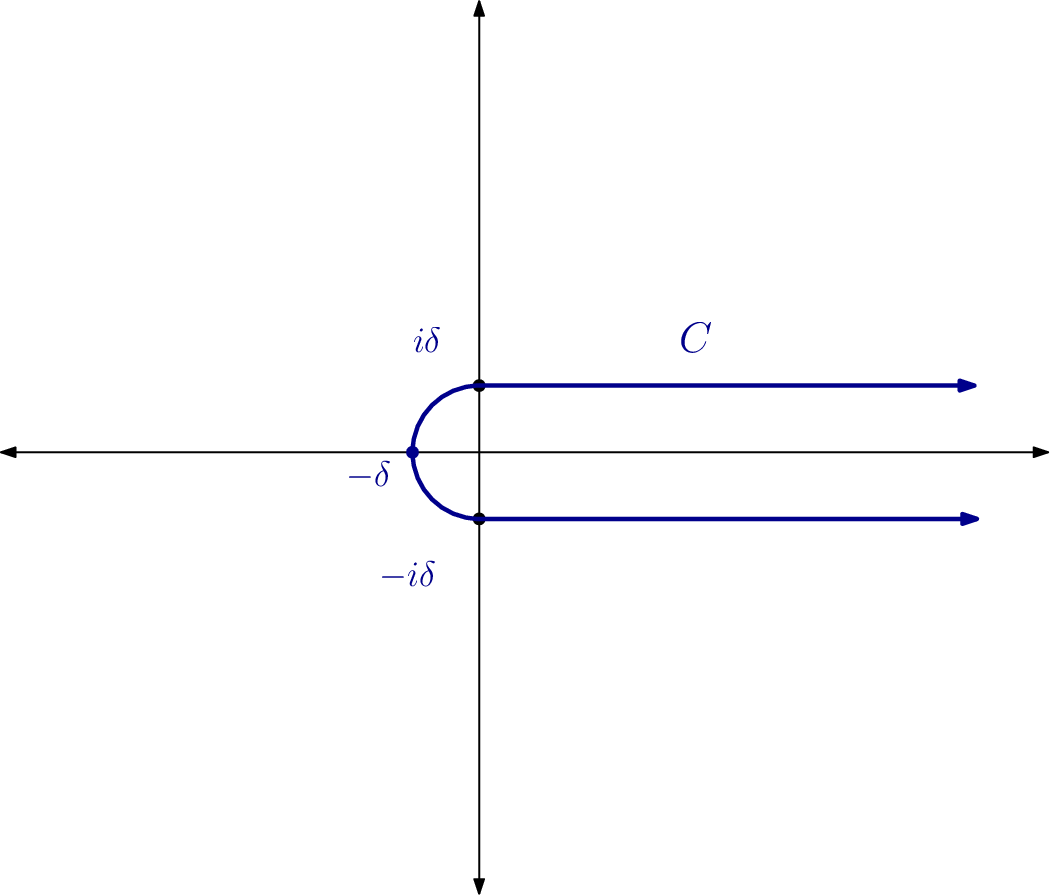}
        \caption{}
        \label{WatsonContour}
    \end{subfigure}%
    ~
    \begin{subfigure}[t]{0.4\textwidth}
        \centering
        \includegraphics[width=0.9\textwidth]{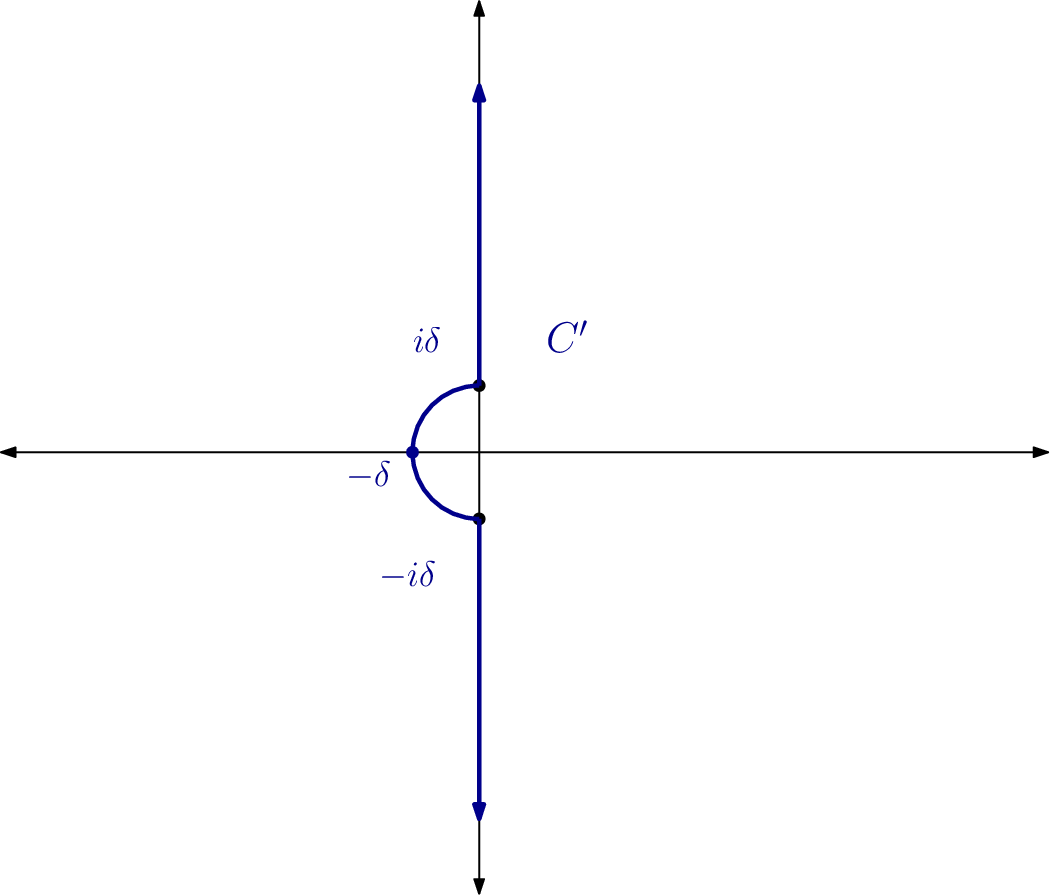}
        \caption{}
        \label{WatsonContourDeformed}
    \end{subfigure}
    \caption{(a) The Watson contour (b) The Watson contour deformed}
\end{figure*}

    \begin{equation}\label{AltSzego}\begin{split}
        &\overline{S_{\mathbb A_r}(z,\zeta)}= \frac{1}{2 \pi i}\int_{C'+1}\frac{r^{2t-1}}{1+r^{2t+1}}\frac{e^{i\pi t}}{e^{2i\pi t}+1}\frac{1}{(z\overline\zeta)^t}dt+ \frac{1}{ 2\pi i}\int_{C'}\frac{1}{1+r^{2t+3}}\frac{e^{i\pi t}}{e^{2i\pi t}+1}(z\overline\zeta)^tdt
        \\&= \frac{1}{2 \pi i}\int_{C'+1}\frac{r^{2t-1}}{1+r^{2t+1}}\frac{1}{\sin(\pi t)}\frac{1}{(z\overline\zeta)^t}dt+ \frac{1}{ 2\pi i}\int_{C'}\frac{1}{1+r^{2t+3}}\frac{1}{\sin(\pi t)}(z\overline\zeta)^tdt.
        \end{split}
    \end{equation}
We chose the representation for $S_{\A_r}(z,\zeta)$ given in (\ref{spectralSzegoAnnulus}) and the associated contour in figure \ref{fundamentalContour} for its connection to the unified transform method as discussed in \cite{C15,C15B,HLLS25} and its connection to the Fourier-Mellin transform as discussed in \cite{C15}. See \cite{GM09} (Chapter 4.2) for a discussion of the Watson transform. It would be interesting to compare the numerical effectiveness of the integral representation for the Szeg\H o kernel given in (\ref{AltSzego}) and (\ref{spectralSzegoAnnulus}).

\section{A remark on the transform pair}
Two annuli, $A_1=\{z:r_1<|z|<R_1\}$ and $A_2=\{z:r_1<|z|<R_2\}$, are conformally equivalent if and only if $r_1/R_1=r_2/R_2$. The Szeg\H o kernel for the annulus $\mathbb A_{r,R}$ is 
\begin{align*}
    \overline{S_{\mathbb A_{r,R}}(z,\zeta)}=\frac{1}{2\pi}\sum_{n=-\infty}^\infty\frac{(z\overline\zeta)^n}{r^{2n+1}+R^{2n+1}}.
\end{align*}
For any annulus $\mathbb A_{r,R}$, there will be a $0<s<1$ such that $\mathbb A_{s}$ is conformally equivalent to $\mathbb A_{r,R}$.
Indeed, consider the conformal map $\Phi(z)=z/R$. The map $\Phi(z)$ will map the annulus $\mathbb A_{r,R}$ to the annulus $\mathbb{A}_{r/R}$. Then by the transformation law \ref{E:Szegokernel} for the Szeg\H o kernel, we have
\begin{align*}
    &\overline{S_{\mathbb A_{r,R}}(z,\zeta)}=\frac{1}{\sqrt{r}}S_{\mathbb A_r}(\Phi(z),\Phi(\zeta))\frac{1}{\overline{\sqrt{r}}}
    =\frac{1}{2\pi r}\sum_{n=-\infty}^{\infty}\frac{(z/r)^n(\overline{\zeta/r})^n}{1+(r/R)^{2n+1}}.
\end{align*}
Thus the transform pair developed in section \ref{NewTransPairSec} holds for any smooth bounded doubly connected domain that is conformally equivalent to the annulus $\mathbb A_{r,R}$ with the substitutions $z\rightarrow z/r$, $\zeta\rightarrow \zeta/r$, and $r\rightarrow r/R$ (and adjusting for the extra factor of $1/r$). All bounded doubly connected domains are conformally equivalent to an annulus or to the punctured disk. A transform pair for punctured disks was developed in \cite{HLLS25}.

\section{Applications}
In this section, the modified Schwarz problem and a point vortex problem for a incompressible flow will be implement via the newly developed transform pair. An eccentric annulus and an elliptical annulus will be the domains of application. The numerical results of the current method for the modified Swcharz problem will be verified against the numerical methods devised by Crowdy in \cite{C15B}. The numerical results of the point vortex problem will be verified against the known solution in order to demonstrate precision.

\subsection{Boundary value problems for doubly connected domains}
Let $u$ be a harmonic function on a doubly connected bounded domain $D$ with smooth boundary curves. We assume that $u$ is a single-valued on $D$. There are numerous boundary value problems to be considered, including, Dirichlet, Nuemann, and Robin boundary value problems. Given boundary data, we may use the transform pair developed in (\ref{transformDoublyConnected}) to numerically solve for a holomorphic function $f=u+iv$ where $v$ is a harmonic conjugate to $u$ on $D$. However, unlike the simply connected case, $f$ is not guaranteed to be single-valued. Functions in $E^2(D)$ though are taken to be single valued. Despite this difficulty, the single-valuedness of $u$ does guarantee the single-valuedness of $f'$ if $\Gamma_1$ and $\Gamma_2$ are assumed to be circular. See \cite{DC09} and references therein for a discussion of this fact. However, this result easily generalizes to any bounded domain that is conformally equivalent to an annulus. In particular, consider the case when $D=\A_r$. Then $f'$ is single-valued, and we have the following transform pair for $f'$:
\begin{equation}\label{transformAnnulusDeriv}
        f'(z)=\frac{1}{2\pi i}\Big[\int_{L_1}\frac{1}{1-e^{2\pi i k}}\tau(z,k)dk+\int_{L_2}\tau(z,k)dk+\int_{L_3}\frac{e^{2\pi  ik}}{1-e^{2\pi  ik}}\tau(z,k)dk\Big],
    \end{equation}
    where
\begin{align*}
        &\rho_{1m}(k)=\int_{\partial  D^m}\frac{f'(\zeta_m)}{\zeta_m^{k+1}}d\zeta_m \hspace{1cm} k\in \C,m=1,2,
        \\& \rho_{2m}(k)=r^{2k+1}\int_{\partial D^m}f'(\zeta_m)\zeta_m^{k}d\zeta_m \hspace{1cm} k\in \C,m=1,2.
    \end{align*}
    % Additionally, define
    % \begin{equation}
    %     \tau(z,k)=\frac{1}{1+r^{2k+1}}\Big[z^k(\rho_{11}(k)+\rho_{22}(-k-1))+\frac{1}{z^{k+1}}(\rho_{21}(k)+\rho_{12}(-k-1))\Big].
    % \end{equation}
One can recover an expression for $f(z)$ by integrate (\ref{transformAnnulusDeriv}) with respect to $z$:
\begin{equation}
        f(z)=\frac{1}{2\pi i}\Big[\int_{L_1}\frac{1}{1-e^{2\pi i k}}\tilde\tau(z,k)dk+\int_{L_2}\tilde\tau(z,k)dk+\int_{L_3}\frac{e^{2\pi  ik}}{1-e^{2\pi  ik}}\tilde\tau(z,k)dk\Big],
    \end{equation}
    where 
    \begin{equation}\label{tildetau}
        \tilde\tau(z,k)=\frac{1}{1+r^{2k+1}}\Big[\frac{z^{k+1}}{k+1}(\rho_{11}(k)+\rho_{22}(-k-1))-\frac{1}{kz^{k}}(\rho_{21}(k)+\rho_{12}(-k-1))\Big].
    \end{equation}
The expressions for $\rho_{1m}$ and $\rho_{2m}$ can also be written in terms of $f$ by applying integration by parts. However, this approach using the general transform pair given in (\ref{transformDoublyConnected}) is far more complicated. One would need to integrate $\sqrt{\Phi'}$ and $\Phi^k$, and some conformal maps may not have a holomorphic anti-derivative. For example, conformal maps with a simple pole may not have a holomorphic anti-derivative. 

\par Next, let us consider the case of Dirichlet boundary values. It is a well known fact that $u$ has a harmonic conjugate $v$ such that $f=u+iv$ is single-valued on $D$ if and only if 
\begin{align*}
     \int_{\Gamma_j}\frac{\partial u}{\partial n}ds=0, \hspace{1cm}1\leq j\leq 2.
 \end{align*}
Note that if $u$ extends to be a harmonic function on the interior of $\Gamma_2$, then the above condition will always be satisfied.
% \red{ See "Conformal mappings of multiply-connected domains" pg. 12 by Jacob S. Huber for a proof.}\color{black}
\subsection{The modified Schwarz problem}
Let $D$ be a doubly connected domain with boundary curves $\Gamma_1$ and $\Gamma_2$ (neither of which are a single point), and let $f=u+iv$ be a holomorphic function on $D$ with prescribed real part $u_1(\zeta)=\Real f(\zeta)$ on $\Gamma_1$ and $u_2(\zeta)=\Real f(\zeta)$ on $\Gamma_2$. The \textbf{modified Schwarz problem} is to find the missing boundary data $v_1(\zeta)=\Imag f(\zeta)$ on $\Gamma_1$ and $v_2(\zeta)=\Imag f(\zeta)$ on $\Gamma_2$, so that $f$ is holomorphic and single-valued on $D$. A closely related variant of this problem is to find the real part of $f$ on $\Gamma_1$ and the imaginary part of on $\Gamma_2$ given the real part of $f$ on $\Gamma_2$ and the imaginary part of $f$ on $\Gamma_1$. 
\subsubsection{Eccentric annulus}
\begin{figure*}[t!]

    \centering
\includegraphics[width=0.7\textwidth]{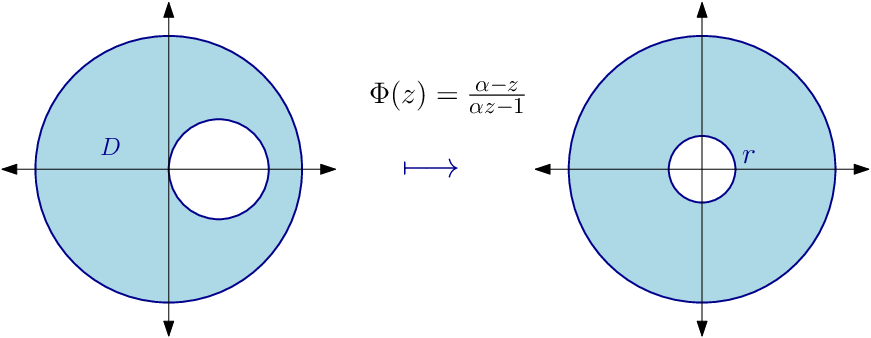}
    
    \caption{An eccentric annular region.}
    \label{EccentricAnnulusGraph}
\end{figure*}

The first application to be considered is the modified Schwarz problem on an eccentric annulus. The current method will be verified against the method developed by Crowdy in \cite{C15B}. The eccentric annulus $D$ under consideration is pictured in figure (\ref{EccentricAnnulusGraph}). Precisely,
\begin{align*}
    D=\{ z\in \C: |z|<1, |z-z_-0|>r_0 \},
\end{align*}
for some $z_0\in\D$ and $0<r_0<\text{dist }(z_0,\partial \D)$. The outer boundary curve, $\Gamma_1$, is the unit circle (centered at the origin), and the inner boundary curve, $\Gamma_2$, is a circle with radius $r_0$ and center $z_0$. We wish to find a conformal map from $D$ to $\mathbb A_r$ for some $r<1$. Such a conformal map will be a linear fraction transformation of the form
\begin{align*}
    \Phi(\zeta)=\frac{\alpha-\zeta}{\alpha \zeta-1},
\end{align*}
where 
\begin{align*}
    &\alpha = \frac{-(r_0^2-z_1^2-1)-\sqrt{(r_0^2-z_1^2-1)^2-4z_0^2}}{2z_0},\quad r=\frac{z_0-r_0-\alpha}{\alpha(z_0-r_0)-1}.
\end{align*}
 Consider the following modified Schwarz problem for $D$. We wish to find an $f\in E^2(D)$ such that
\begin{align*}
    \begin{cases}
\overline\partial f(z)=0, & z\in D\\ \Real f(\zeta)=u_1(\zeta), & \zeta\in \Gamma_1\\
\Real f(\zeta)=u_2(\zeta), & \zeta\in \Gamma_2
    \end{cases}
\end{align*}    
The goal is two-fold. First, find the imaginary part of $f$ on $\Gamma_1$, denoted by $v_1$, and the imaginary part of $f$ on $\Gamma_2$, denoted by $v_2$. Second, recover the interior values of $f$ by using (\ref{transformDoublyConnected}). The missing boundary data will be found using the global relation given in (\ref{GR}):
 \begin{align*}
      \rho_{j1,\Phi}(k)+\rho_{j2,\Phi}(k)=0, \hspace{5mm} k\in \mathbb N\cup \{0\}, j=1,2.
    \end{align*}
On $\Gamma_1=\partial\mathbb D$ we write 
\begin{align*}
    v_1(\zeta)=a_0+\sum_{n=1}^\infty a_n\zeta^n+\sum_{n=1}^\infty \overline{a_n}\overline{\zeta^n},
\end{align*}
where the coefficients $a_0 \in \mathbb{R}$ and $\{a_n \in \mathbb C | n=1,2,\dots \}$ are to be determined. On $\Gamma_2=\partial D_{r_0}(z_0)$, we write
\begin{align*}
    v_2(\zeta)=b_0+\sum_{n=1}^\infty b_n\Big(\frac{\zeta-z_0}{r_0}\Big)^n+\sum_{n=1}^\infty \overline {b_n}\Big(\frac{\overline{\zeta-z_0}}{r_0}\Big)^n,
\end{align*}
where the coefficients $b_0 \in \mathbb{R}$ and $\{b_n \in \mathbb C | n=1,2,\dots \}$ are to be determined. We choose the standard parametrization for $\Gamma_1$, $\zeta_1(\theta)=e^{i\theta}$, with $\theta\in[0,2\pi]$.
We then have
\begin{equation}\label{Seriesan}
f(\zeta_1(\theta))=u(\zeta_1(\theta))+\ci\Big(a_0+\sum_{n\geq 1}\Big[a_n\ee^{\ci n\theta}+\overline{a_n} \ee^{- \ci n\theta}\Big]\Big).  
\end{equation}
Likewise we choose the standard parametrization for $\Gamma_2$, $\zeta_2(\theta)=z_0+r_0e^{i\theta}$, with $\theta\in[0,2\pi]$. We then have
\begin{equation}\label{seriesbn}
f(\zeta_2(\theta))=u(\zeta_2(\theta))+\ci\Big(b_0+\sum_{n\geq 1}\Big[b_n\ee^{\ci n\theta}+\overline{b_n} \ee^{- \ci n\theta}\Big]\Big). 
\end{equation}
{\color{blue}}
We have two global relations. For $j=1$
\begin{align*}
  &0=  \rho_{11,\Phi}(k)+\rho_{12,\Phi}(k)=\int_{\Gamma_1}\frac{f(\zeta)\sqrt{\Phi'(\zeta)
        }}{\Phi(\zeta)^{k+1}}d\zeta +\int_{\Gamma_2}\frac{f(\zeta)\sqrt{\Phi'(\zeta)
        }}{\Phi(\zeta)^{k+1}}d\zeta 
        \\&=\int_{0}^{2\pi}\frac{\sqrt{\Phi'(\zeta_1(\zeta))
        }}{\Phi(\zeta_1(\theta))^{k+1}}\Big(u(\zeta_1(\theta))+\ci\Big(a_0+\sum_{n\geq 1}\Big[a_n\ee^{\ci n\theta}+\overline{a_n} \ee^{- \ci n\theta}\Big]\Big)\Big)\zeta_1'(\theta)d\theta
        \\&-\int_{0}^{2\pi}\frac{\sqrt{\Phi'(\zeta_2(\zeta))
        }}{\Phi(\zeta_2(\theta))^{k+1}}\Big(u(\zeta_2(\theta))+\ci\Big(b_0+\sum_{n\geq 1}\Big[b_n\ee^{\ci n\theta}+\overline{b_n} \ee^{- \ci n\theta}\Big]\Big)\Big)\zeta_2'(\theta)d\theta.
\end{align*}
We may rewrite the equation above as

\begin{equation}\label{EccAnnEqns1}
\begin{split}
&a_0\cA(0,k)+\sum_{n=1}^\infty\Big(a_n\cA(n,k)+\overline{a_n} \cA(-n,k)\Big) \\
&~~~~~~+b_0\cB(0,k)+\sum_{n=1}^\infty\Big(b_n \cB(n,k)+\overline{b_n} \cB(-n,k)\Big)=s(k), \quad k \in \mathbb{N}\cup\{0\}, \\
&a_0\overline{\cA(0,k)}+\sum_{n=1}^\infty\Big(\overline{a_n} \overline{\cA(n,k)}+ a_n\overline{\cA(-n,k)}\Big) \\
&~~~~~~+b_0\overline{\cB(0,k)} +\sum_{n=1}^\infty\Big(\overline{b_n} \overline{\cB(n,k)}+ b_n\overline{\cB(-n,k)}\Big)=\overline{s(k)}, \quad k \in \mathbb{N}\cup\{0\},
\end{split}
\end{equation}
where
\begin{align*}
    \cA(n,k)=i\int_0^{2\pi}\frac{\sqrt{\Phi'(\zeta_1(\zeta))
        }}{\Phi(\zeta_1(\theta))^{k+1}}e^{in\theta}\zeta_1'(\theta)d\theta,\quad \cB(n,k)=-i\int_0^{2\pi}\frac{\sqrt{\Phi'(\zeta_2(\zeta))
        }}{\Phi(\zeta_2(\theta))^{k+1}}e^{in\theta}\zeta_2'(\theta)d\theta,
\end{align*}
and
\begin{align*}
    s(k)=-\int_0^{2\pi}\frac{\sqrt{\Phi'(\zeta_1(\zeta))
        }}{\Phi(\zeta_1(\theta))^{k+1}}u(\zeta_1(\theta))\zeta_1'(\theta)d\theta+\int_0^{2\pi}\frac{\sqrt{\Phi'(\zeta_2(\zeta))
        }}{\Phi(\zeta_2(\theta))^{k+1}}u(\zeta_2(\theta))\zeta_2'(\theta)d\theta.
\end{align*}
 Likewise, with $k \in \mathbb{N}\cup\{0\}$, the global relation for $j=2$ gives

\begin{equation}\label{EccAnnEqns2}
\begin{split}
&a_0\cA(0,-k-1)+\sum_{n=1}^\infty\Big(a_n\cA(n,-k-1)+\overline{a_n} \cA(-n,-k-1)\Big) \\
&~~~~~~+b_0\cB(0,-k-1)+\sum_{n=1}^\infty\Big(b_n \cB(n,-k-1)+\overline{b_n} \cB(-n,-k-1)\Big)=s(-k-1), \\
&a_0\overline{\cA(0,-k-1)}+\sum_{n=1}^\infty\Big(\overline{a_n} \overline{\cA(n,-k-1)}+ a_n\overline{\cA(-n,-k-1)}\Big) \\
&~~~~~~+b_0\overline{\cB(0,-k-1)} +\sum_{n=1}^\infty\Big(\overline{b_n} \overline{\cB(n,-k-1)}+ b_n\overline{\cB(-n,-k-1)}\Big)=\overline{s(-k-1)}.
\end{split}
\end{equation}
Next, Crowdy's method is implemented in a similar manner. The first global relation is
\begin{align*}
    0=\int_{\partial \D}f(\zeta)z^{-k-1}dz-\int_{\partial D_{r_0}(z_0)} f(z)z^{-k-1}dz, \hspace{1cm} k\in -\mathbb N.
\end{align*}

Using the same parameterizations $\zeta_1(\theta)$ and $\zeta_2(\theta)$ as above, we have
\begin{equation}\label{EccAnnEqnsCr}
\begin{split}
&a_0{\cA}'(0,k)+\sum_{n=1}^\infty\Big(a_n{\cA}'(n,k)+\overline{a_n} {\cA}'(-n,k)\Big) \\
&~~~~~~+b_0{\cB}'(0,k)+\sum_{n=1}^\infty\Big(b_n {\cB}'(n,k)+\overline{b_n} {\cB}'(-n,k)\Big)=s'(k), \quad k \in \mathbb{N}\cup\{0\}, \\
&a_0\overline{{\cA}'(0,k)}+\sum_{n=1}^\infty\Big(\overline{a_n} \overline{{\cA}'(n,k)}+ a_n\overline{{\cA}'(-n,k)}\Big) \\
&~~~~~~+b_0\overline{{\cB}'(0,k)} +\sum_{n=1}^\infty\Big(\overline{b_n} \overline{{\cB}'(n,k)}+ b_n\overline{{\cB}'(-n,k)}\Big)=\overline{s'(k)}, \quad k \in \mathbb{N}\cup\{0\},
\end{split}
\end{equation}
where
\begin{align*}
    {\cA}'(n,k)=i\int_0^{2\pi}\frac{e^{in\theta}\zeta_1'(\theta)}{(\zeta_1(\theta))^{k+1}}d\theta,\quad {\cB}'(n,k)=-i\int_0^{2\pi}\frac{e^{in\theta}\zeta_2'(\theta)}{(\zeta_2(\theta))^{k+1}}d\theta,
\end{align*}
and
\begin{align*}
    s'(k)=-\int_0^{2\pi}\frac{u(\zeta_1(\theta))\zeta_1'(\theta)}{(\zeta_1(\theta))^{k+1}}d\theta+\int_0^{2\pi}\frac{u(\zeta_2(\theta))\zeta_2'(\theta)}{(\zeta_2(\theta))^{k+1}}d\theta.
\end{align*}
Likewise, the corresponding equations can be generated for the other global relation. 

Now, the solution of this boundary value problem will be unique up to a constant, so we may assume without loss of generality that $a_0=0$ and $b_0=0$. So we may omit the $n=0$ equation from the set of equations above. The series (\ref{Seriesan}) and (\ref{seriesbn}) are truncated up to $n=N$ to include only a finite number of unknown coefficients. Then (\ref{EccAnnEqns1}) and (\ref{EccAnnEqns2}) generate a system of equations with $4N$ unknowns. Then evaluating (\ref{EccAnnEqns1}) and (\ref{EccAnnEqns2}) at $k=1,2,\dots,K$ generates a system of equations (over-determined if $K>N$) that is solved with a least-squares algorithm. The following example was implement in Matlab with the following parameters: $r_0=1/3, z_0=1/4, u_1(\zeta)=\Real(e^\zeta), u_2(\zeta)=\Real((\zeta-z_0)^{-2})$. The coefficients $\{a_n,b_n: n\geq 1\}$ decay quickly for both the current method and for the method presented in \cite{C15B}, so $N$ was chosen to be $N=8$ and $K=12$. Figure \ref{EccentricAnnulusBoundaryVlues} shows the the numerical solutions for $v_1(\zeta_1(\theta))$ and $v_2(\zeta_2(\theta))$ are comparable for both the current method and for the method presented in \cite{C15B}.

\begin{figure}
\centering
\includegraphics[width=0.7\textwidth]{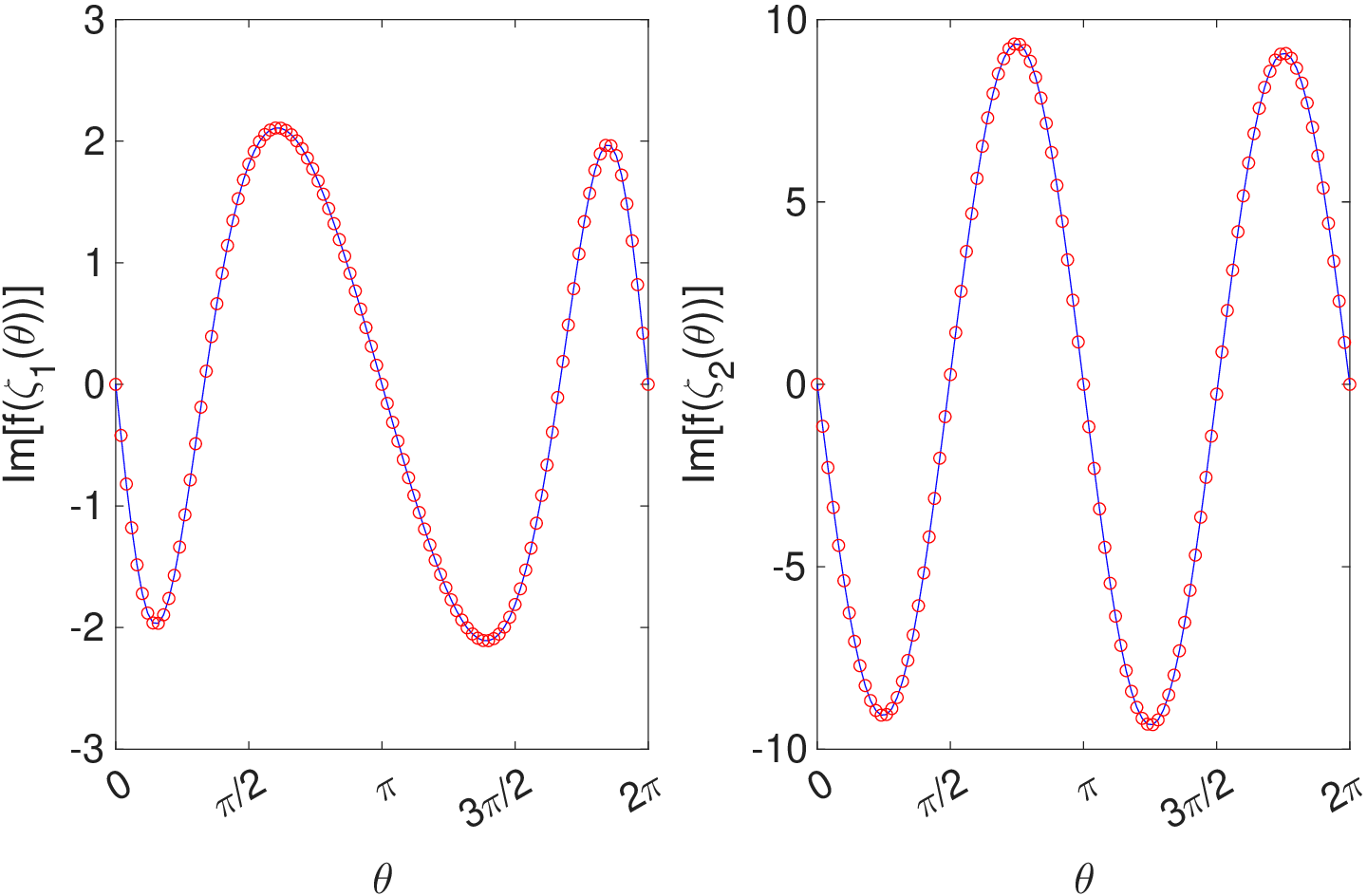}
\caption{The current method is graphed with a solid blue line and the method in \cite{C15B} is graphed with red dots.}
\label{EccentricAnnulusBoundaryVlues}
\end{figure}

\subsection{Elliptical annular region}\label{ellipAnnSec}

\begin{figure*}[t!]

    \centering
\includegraphics[width=0.9\textwidth]{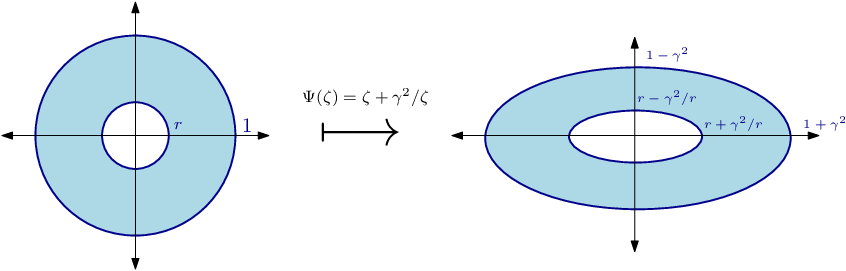}
    
    \caption{An elliptical annular region.}
    \label{EllipticalAnnulusGraph}
\end{figure*}
Next, a variant of the modified Schwarz problem will be implemented on a doubly connected domain $D$ whose boundary curves are non-circular. Define the domain $D$ as

\begin{align*}
    D:=\Big\{z=x+iy: \frac{x^2}{(1+\gamma^2)^2}+\frac{y^2}{(1-\gamma^2)^2}<1,\frac{x^2}{(r+\gamma^2/r)^2}+\frac{y^2}{(r-\gamma^2/r)^2}>1 \Big\},
\end{align*}
where $0<\gamma<r<1$ as pictured in (\ref{EllipticalAnnulusGraph}). The holomorphic function $\Phi:D\rightarrow \A_r$, $\Phi(\zeta)=\frac{1}{2}(\zeta+\sqrt{\zeta^2-4\gamma^2})$, will conformally map $D$ onto $\A_r$.  The inverse of $\Phi$ is $\Phi^{-1}(\zeta)=\Psi(\zeta)=\zeta+\gamma^2/\zeta$. Note that $\Phi$ and its derivative are single-valued on $D$. Indeed, the two branch points of $\Phi$, $\zeta^2-4\gamma^2=0$, need to lie in the same component of the complement of $D$. However, $2\gamma<r+\gamma^2/r \Longleftrightarrow 0<(r-\gamma)^2$, hence both lie within the bounded component of the complement of $D$. Let $\Gamma_1$ denote the outer boundary of $D$: 
\begin{align*}
   \Gamma_1= \Big\{z=x+iy:\frac{x^2}{(1+\gamma^2)^2}+\frac{y^2}{(1-\gamma^2)^2}=1\Big\},
\end{align*}
and let $\Gamma_2$ denote the inner boundary of $D$: 
\begin{align*}
\Gamma_2=\Big\{z=x+iy: \frac{x^2}{(r+\gamma^2/r)^2}+\frac{y^2}{(r-\gamma^2/r)^2}=1\Big\}.\end{align*}
We wish to find an $f\in E^2(D)$ such that

\begin{align*}
    \begin{cases}
\overline\partial f(z)=0, & z\in D\\ \Real f(\zeta)=u_1(\zeta), & \zeta\in \Gamma_1\\
\Imag f(\zeta)=v_2(\zeta), & \zeta\in \Gamma_2
    \end{cases}
\end{align*}
where $u_1$ and $v_2$ are known on the boundary. The transform method presented in this work will be used to solve this problem in the following manner. First, complete the missing boundary data: find $\Imag f=v_1$ on $\Gamma_1$ and $\Real f=u_2$ on $\Gamma_2$. Second, recover the interior values of $f$ by using (\ref{transformDoublyConnected}). The missing boundary data will be found by using the global relations given in (\ref{GR}).
The global relation given in (\ref{GR}) involves evaluating $\sqrt{\Phi'}$ along $\partial \D$. We remark that an equivalent statement of analyticity given by the global relation (\ref{GR}) is given by the pullback of (\ref{GR}) to $\A_r$:

\begin{align}\label{GRPullBack}
\begin{split}
  0&=\int_{\partial D}\frac{f(w)\sqrt{\Phi'(w)}}{\Phi(w)^{k+1}}dw =\int_{\partial \A_r}\frac{f(\Psi(\zeta))\sqrt{\Phi'(\Psi(\zeta))
        }\Psi'(\zeta)}{\Phi(\Psi(\zeta))^{k+1}}d\zeta
        \\&= \int_{\partial \D}\frac{f(\Psi(\zeta))\sqrt{\Psi'(\zeta)
        }}{\zeta^{k+1}}d\zeta+\int_{\partial D_r(0)}\frac{f(\Psi(\zeta))\sqrt{\Psi'(\zeta)
        }}{\zeta^{k+1}}d\zeta.
        \end{split}
\end{align}
For certain domains or applications, one may choose to use the equivalent global relation above for simplicity or practicality. We will use (\ref{GRPullBack}) to analyze this boundary value problem as certain expressions were shorter when working with $\Psi$ as oppose to $\Phi$.
For $\zeta\in \Gamma_1$, write
\begin{align}\label{v1Rep}
v_1(\zeta)=a_0+\sum_{n=1}^\infty a_n  \Phi(\zeta)^n+\sum_{n=1}^\infty \overline{a_n}  \overline{\Phi(\zeta)^n}.
\end{align}
For $\zeta\in \Gamma_2$, write
\begin{align}\label{v2Rep}
u_2(\zeta)=b_0+\sum_{n=1}^\infty b_n  \Phi(\zeta)^n+\sum_{n=1}^\infty \overline{b_n}  \overline{\Phi(\zeta)^n}.
\end{align}
Next, consider the parametrization of $\Gamma_1$ given by the conformal map $\Psi$:
\begin{align}\label{para1}
    \theta\mapsto\Psi\Big|_{\partial \D}(e^{i\theta})=e^{i\theta}+\gamma^2/e^{i\theta},\quad \theta\in[0,2\pi]
\end{align}
and the parametrization of $\Gamma_2$ given by
\begin{align}\label{para2}
    \theta\mapsto\Psi\Big|_{\partial D_r(0)}(r e^{i\theta})=r e^{i\theta}+\frac{\gamma^2}{r e^{i\theta}},\quad \theta\in[0,2\pi].
\end{align}
It follows that
\begin{align*}
    &v_1(\Psi(e^{i\theta}))=a_0+\sum_{n=1}^\infty a_n e^{in\theta}+\sum_{n=1}^\infty \overline{a_n} e^{-in\theta}
\end{align*}
or equivalently
\begin{align}\label{v1ParaZeta}
    v_1(\Psi(\zeta))=a_0+\sum_{n\geq 1}a_n\zeta^n+\sum_{n\geq 1}\overline {a_n}\overline{\zeta^n}=a_0+\sum_{n\geq 1}a_n\zeta^n+\sum_{n\geq 1}\overline {a_n}\zeta^{-n},\hspace{1cm} \zeta \in \partial \D
\end{align}
and
\begin{align*}
    u_2(\Psi(r e^{i\theta}))=b_0+\sum_{n=1}^\infty b_n r^n e^{in\theta}+\sum_{n=1}^\infty \overline{b_n}r^n e^{-in\theta}  
\end{align*}
or equivalently
\begin{align}\label{v2ParaZeta}
    u_2(\Psi(\zeta))=b_0+\sum_{n\geq 1}b_n\zeta^n+\sum_{n\geq 1}\overline {b_n}\ \overline{\zeta^n}=b_0+\sum_{n\geq 1}b_n\zeta^n+\sum_{n\geq 1}\overline{b_n}r^{2n}\zeta^{-n} \hspace{1cm} \zeta\in\partial D_r(0).
\end{align}
\textbf{Remark:} When evaluating $f$ at an interior point $z\in D$ using the expression given in (\ref{transformDoublyConnected}), choosing the parametrization for $\partial \D$ and $\partial D_r(0)$ given in (\ref{para1}) and (\ref{para2}) is simplest because of the choice of representation of the functions $v_1$ and $u_2$ given in (\ref{v1Rep}) and (\ref{v2Rep}).

Next, substituting the representation for $v_1$ and $u_2$ found in (\ref{v1ParaZeta}) and (\ref{v2ParaZeta}) into the equivalent global relation (for $j=1$) given in (\ref{GRPullBack}) yields
\begin{equation}\label{ellipticEqns}
\begin{split}
&a_0\cA(0,k)+\sum_{n=1}^\infty\Big(a_n\cA(n,k)+\overline{a_n} \cA(-n,k)\Big) \\
&~~~~~~+b_0\cB(0,0,k)+\sum_{n=1}^\infty\Big(b_n \cB(n,k)+\overline{b_n}r^{2n} \cB(-n,k)\Big)=s(k), \quad k \in \mathbb{Z}, \\
&a_0\overline{\cA(0,k)}+\sum_{n=1}^\infty\Big(\overline{a_n} \overline{\cA(n,k)}+ a_n\overline{\cA(-n,k)}\Big) \\
&~~~~~~+b_0\overline{\cB(0,k)} +\sum_{n=1}^\infty\Big(\overline{b_n} \overline{\cB(n,k)}+ b_nr^{2n}\overline{\cB(-n,k)}\Big)=\overline{s(k)}, \quad k \in \mathbb{Z},
\end{split}
\end{equation}
where
\begin{align*}
    \cA(n,k)=i\int_{\partial \D}\frac{\sqrt{\Psi'(\zeta)}}{\zeta^{k-n+1}}d\zeta, \quad \cB(n,k)=\int_{\partial D_{r}(0)}\frac{\sqrt{\Psi'(\zeta)}}{\zeta^{k-n+1}}d\zeta,
\end{align*}
and
\begin{align*}
    s(k)=-i\int_{\partial \D}\frac{\sqrt{\Psi(\zeta)
        }}{\zeta^{k+1}}v_1(\Psi(\zeta))d\zeta+\int_{\partial D_r(0)}\frac{\sqrt{\Psi(\zeta)
        }}{\zeta^{k+1}}u_2(\Psi(\zeta))d\zeta.
\end{align*}
A second set of equations can similarly be generated for the second global relation ($j=2$). Now, the solution of this boundary value problem will be unique up to a constant, so we may assume WLOG that $a_0=0$ and $b_0=0$. So we may omit the $n=0$ equation from the set of equations above. The series (\ref{v1ParaZeta}) and (\ref{v2ParaZeta}) are truncated up to $n=N$ to include only a finite number of unknown coefficients. Then (\ref{ellipticEqns}) and the corresponding set of equations for $j=2$ generate a system of equations with $4N$ unknowns. Then evaluating (\ref{ellipticEqns})  and the corresponding set of equations for $j=2$ at $k=1,2,\dots,K$ generates an over-determined (if $K>N$) system of equations that is solved with a least-squares algorithm. The following example was implement in Matlab with the following parameters: $\gamma=1/4, r=3/4$, $u_1(\zeta)=\Real(\zeta^{3})$, $v_2(\zeta)=\Imag(\zeta^{-1})$. The coefficients $\{a_n,b_n: n\geq 1\}$ decay quickly, so $N$ was chosen to be $N=8$ and $K=8$. Figure \ref{EllipticAnnulusRI} shows $v_1(\Psi(e^{i\theta}))$ and $u_2(\Psi(r e^{i\theta}))$.

\begin{figure}
\centering
\includegraphics[width=0.5\textwidth]{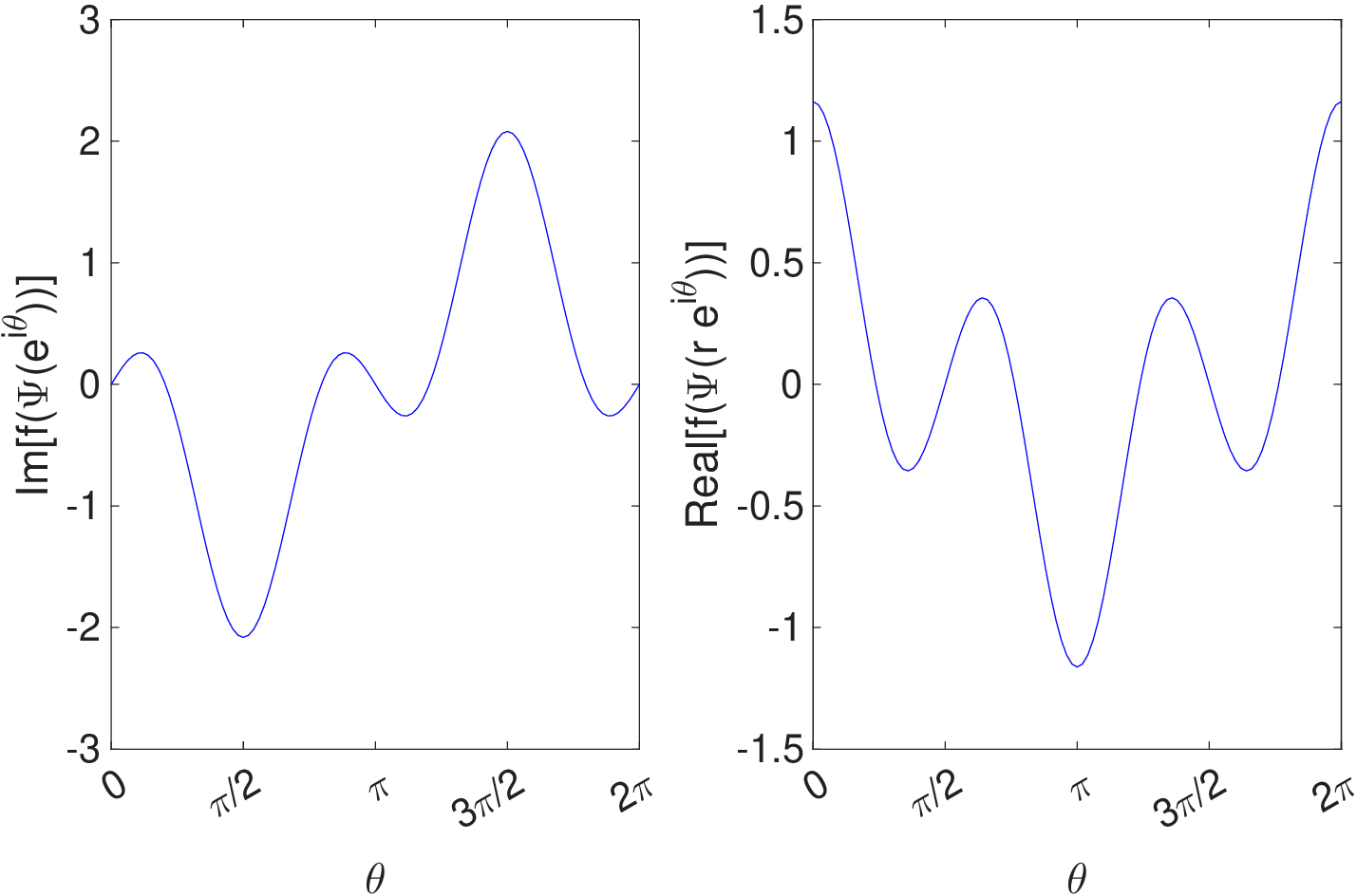}
\caption{The graphs of $v_1(\Psi(e^{i\theta}))$ and $u_2(\Psi(r e^{i\theta}))$. }
\label{EllipticAnnulusRI}
\end{figure}

\begin{figure}
\centering
\includegraphics[width=0.45\textwidth]{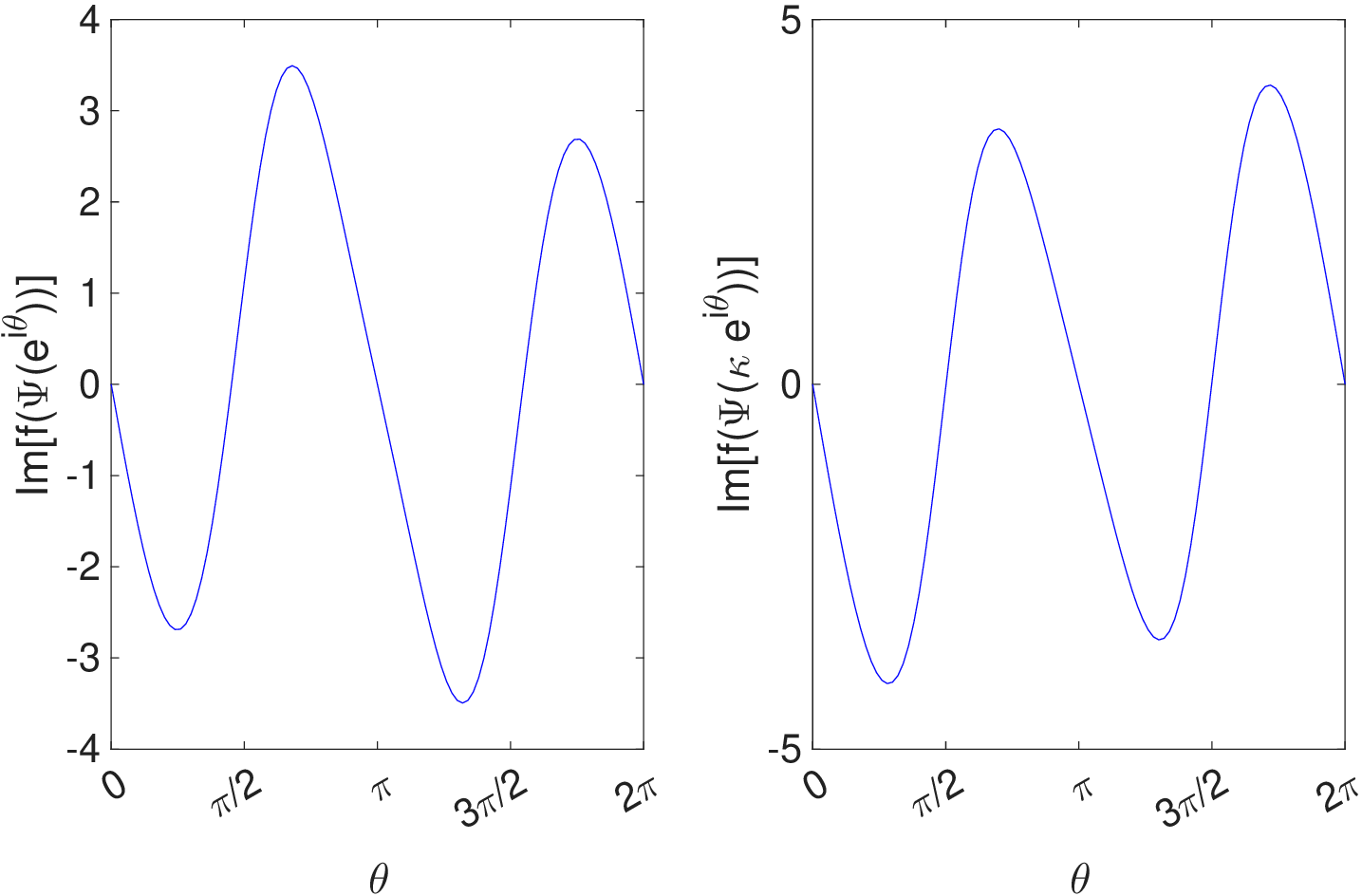}
\caption{The graphs of $v_1(\Psi(e^{i\theta}))$ and $u_2(\Psi(r e^{i\theta}))$. }
\label{EllipticAnnulusRR}
\end{figure}
The classical Schwarz problem of finding an $f\in E^2(D)$ such that
\begin{align*}
    \begin{cases}
\overline\partial f(z)=0, & z\in D\\ \Real f(\zeta)=u_1(\zeta), & \zeta\in \Gamma_1\\
\Real f(\zeta)=u_2(\zeta), & \zeta\in \Gamma_2
    \end{cases}
\end{align*} 
is implemented in a very similar fashion. Following a similar procedure as was outlined for the boundary value problem above, one can numerically recover the missing boundary data $v_1(\Phi(\zeta))$ and $u_2(\Phi(\zeta))$. With $\gamma=1/4$ and $r=3/4$ as before and 
$u_1(\zeta)=\Real(e^{-\zeta}+z^{-1}), u_2(\zeta)=\Real(\zeta^2+\zeta^{-2})$. Figure (\ref{EllipticAnnulusRR}) shows a graph of  $v_1(\Phi(\zeta))$ and $u_2(\Phi(\zeta))$ when implemented using Matlab.

\subsection{Fluid dynamics: a point vortex problem}
Let $D$ be the elliptical domain defined in section \ref{ellipAnnSec}, and let $W_D$ be the complex potential for an incompressible steady flow that is irrotational except at a point vortex $\alpha \in D$ (say of circulation 1). For the purposes of this example, we will also assume that the there is a zero-circulation condition on the inner boundary curve.  Because the flow is irrotational (except at $\alpha$), there will be a velocity potential $\phi_\alpha$ for the flow that is harmonic on $D$, and a harmonic conjugate to $\phi_\alpha$, denoted by $\psi_\alpha$, that is the stream function for  the flow. We then have the following decomposition for $W_D$:
\begin{align*}
   & W_D(w,\alpha)=\frac{1}{2\pi i}\log(\Phi(w)-\Phi(\alpha))+h_D(z,\alpha), \quad w \in \overline{D}
\end{align*} 
where $h_D(z,\alpha):=\phi_\alpha(z)+i\psi_\alpha(z)$. It should be noted that $h_D$ is a single-valued holomorphic function on $D$ that extends continuously to the boundary. See  \cite{MC05} or the discussion below for the justification of this decomposition. We remark that given the conformally invariant nature of this problem, this setup is applicable to more general doubly connected domains. The boundaries of the domain $D$ need to be streamlines, i.e. $\text{Im}(W(z,\alpha))$ must be constant on the boundary curves $\Gamma_1$ and $\Gamma_2$. For the purposes of this example, we require the following impermeability condition on the outer boundary $\Gamma_1$:
\begin{align*}
    \text{Im}(W(\zeta,\alpha))=0,\quad \zeta\in\Gamma_1.
\end{align*}
As mentioned previously, we also require the net circulation around the inner boundary curve to be zero, hence 
\begin{align*}
    \text{Im}(W(\zeta,\alpha))=0,\quad \zeta\in\Gamma_2.
\end{align*}
 This problem has a known exact solution, and our results will be verified against the exact solution. The boundary conditions above imply that
\begin{align*}
    &\text{Im}\Big[\frac{1}{2\pi i}\log(\Phi(\zeta)-\Phi(\alpha))\Big]=-\text{Im}[h_D(\zeta)], \quad \zeta \in \Gamma_1,
    \\&
    \\& \text{Im}\Big[\frac{1}{2\pi i}\log(\Phi(\zeta)-\Phi(\alpha))\Big]=-\text{Im}[h_D(\zeta)], \quad \zeta \in \Gamma_2. 
\end{align*}
For $\zeta \in \Gamma_1$, write
\begin{align}\label{vortex1Rep}
\text{Re}(h_D(\zeta))=\sum_{n=1}^\infty a_n  \Phi(\zeta)^n+\sum_{n=1}^\infty \overline{a_n}  \overline{\Phi(\zeta)^n}.
\end{align}
For $\zeta\in \Gamma_2$, write
\begin{align}\label{vortex2Rep}
\text{Re}(h_D(\zeta))=\sum_{n=1}^\infty b_n  \Phi(\zeta)^n+\sum_{n=1}^\infty \overline{b_n}  \overline{\Phi(\zeta)^n}.
\end{align}
Thus
\begin{align*}
    &h_D(\Psi(e^{i\theta}))=\sum_{n=1}^\infty a_n e^{in\theta}+\sum_{n=1}^\infty \overline{a_n} e^{-in\theta} -i\text{Im}\Big[\frac{1}{2\pi i}\log(e^{i\theta}-\alpha)\Big],\quad \theta\in[0,2\pi]
    \end{align*}
    for $\Psi(e^{i\theta})=\zeta\in\Gamma_1$ and
    \begin{align*}
    \\&h_D(\Psi(re^{i\theta}))=\sum_{n=1}^\infty b_n r^n e^{in\theta}+\sum_{n=1}^\infty \overline{b_n}r^n e^{-in\theta}  -i\Big( \text{Im}\Big[\frac{1}{2\pi i}\log(re^{i\theta}-\alpha)\Big]\Big),\quad \theta\in[0,2\pi]
\end{align*}
for $\Psi(re^{i\theta})=\zeta\in\Gamma_2$. As before, we use the global relations as formulated in equation (\ref{GRPullBack}).
For the first global relation ($j=1$) our linear system simplifies to
\begin{equation}\label{vortexEqns}
\begin{split}
&\sum_{n=1}^\infty\Big(a_n\cA(n,k)+\overline{a_n} \cA(-n,k)\Big) \\
&~~~~~~+\sum_{n=1}^\infty\Big(b_n \cB(n,k)+\overline{b_n}r^{2n} \cB(-n,k)\Big)=s(k), \quad k \in \mathbb{Z}, \\
&\sum_{n=1}^\infty\Big(\overline{a_n} \overline{\cA(n,k)}+ a_n\overline{\cA(-n,k)}\Big) \\
&~~~~~~ +\sum_{n=1}^\infty\Big(\overline{b_n} \overline{\cB(n,k)}+ b_nr^{2n}\overline{\cB(-n,k)}\Big)=\overline{s(k)}, \quad k \in \mathbb{Z},
\end{split}
\end{equation}
where
\begin{align*}
    \cA(n,k)=\int_{\partial \D}\frac{\sqrt{\Psi'(\zeta)}}{\zeta^{k-n+1}}d\zeta,\quad \cB(n,k)=-\int_{\partial D_{r}(0)}\frac{\sqrt{\Psi'(\zeta)}}{\zeta^{k-n+1}}d\zeta
\end{align*}
and
\begin{align*}
    s(k)=i\int_{\partial \D}\frac{\sqrt{\Psi(\zeta)
        }}{\zeta^{k+1}}\text{Im}\Big[\frac{1}{2\pi i}\log(\zeta-\Phi(\alpha))\Big]d\zeta-i\int_{\partial D_r(0)}\frac{\sqrt{\Psi(\zeta)
        }}{\zeta^{k+1}}\text{Im}\Big[\frac{1}{2\pi i}\log(\zeta-\Phi(\alpha))\Big]d\zeta.
\end{align*}
A second set of equations can similarly be generated for the second global relation ($j=2$).  The series (\ref{vortex1Rep}) and (\ref{vortex2Rep}) are truncated up to $n=N$ to include only a finite number of unknown coefficients. Then (\ref{ellipticEqns}) and the corresponding set of equations for $j=2$ generate a system of equations with $4N$ unknowns. Then evaluating (\ref{ellipticEqns})  and the corresponding set of equations for $j=2$ at $k=1,2,\dots,K$ generates an over-determined (if $K>N$) system of equations that is solved with a least-squares algorithm. 

\par Next we verify our numerical solution against the exact solution. Let $W_{\mathbb A_r}(z,\alpha)$ be the complex potential for a flow with vortex at $\alpha_0\in \mathbb A_r$. The exact solution for for a flow past two circular cylinders was first formulated by Johnson \& McDonald in 2004 \cite{JM04}, and was later generalized to finitely connected domains by Marshall \& Crowdy in 2005 \cite{MC05}. Define
\begin{align*}
    P(\zeta,r)=(1-\zeta)\prod_{k=1}^\infty (1-r^{2k}\zeta)(1-r^{2k}\zeta^{-1}).
\end{align*}
Then, as formulated in \cite{MC05},
\begin{align*}
    &W_{\mathbb A_r}(\zeta,\alpha_0)=-\frac{i}{2\pi}\log\Big(|\alpha_0|\frac{P(\zeta \alpha_0^{-1},r)}{P(\zeta\overline{\alpha_0},r)}\Big)
    =\frac{1}{2\pi}\text{Arg}\Big(|\alpha_0|\frac{P(\zeta \alpha_0^{-1},r)}{P(\zeta\overline{\alpha_0},r)}\Big)+\frac{1}{2\pi i}\ln\Big| \alpha_0\frac{P(\zeta \alpha^{-1},r)}{P(\zeta\overline{\alpha_0},r)}\Big|.
\end{align*}
Thus the streamline function is
\begin{align*}
    \Psi_{\alpha_0}(\zeta)=-\frac{1}{2\pi}\ln\Big| \alpha_0\frac{P(\zeta \alpha_0^{-1},r)}{P(\zeta\overline{\alpha_0},r)}\Big|.
\end{align*}
Next, consider the zeros of the numerator and denominator of $P(\zeta \alpha_0^{-1},r)/P(\zeta\overline{\alpha_0},r)$ appearing in the expression below,
\begin{align*}
    \log\Big(|\alpha_0|\frac{P(\zeta \alpha_0^{-1},r)}{P(\zeta\overline{\alpha_0},r)}\Big)=\log\Big(\frac{|\alpha_0|}{\alpha_0}\frac{(\alpha_0-\zeta )\prod_{k=1}^\infty (1-r^{2k}\zeta \alpha_0^{-1})(1-r^{2k}\zeta^{-1} \alpha_0)}{(1-\zeta\overline{\alpha_0})\prod_{k=1}^\infty (1-r^{2k}\zeta\overline{\alpha_0})(1-r^{2k}\zeta^{-1}(\overline{\alpha_0})^{-1})}\Big).
\end{align*}
Now, the zeros of the numerator that lie inside $|z|<r$ are $r^{2k}\alpha_0$ for $k\geq 1$ and the zeros of the denominator that lie in in $|z|<r$ are $r^{2k}/\overline{\alpha_0}$. Thus the complex potential $W$ can be written as
\begin{align*}
    W_{\mathbb A_r}(\zeta,\alpha_0)=\frac{1}{2\pi i}\log(z-\alpha_0)+h_{\mathbb A_r}(\zeta,\alpha_0)
\end{align*} 
where 
\begin{align*}
    h_{\mathbb A_r}(z,\alpha):=\frac{1}{2\pi i}\log\Big(\frac{|\alpha_0|}{\alpha_0}\frac{\prod_{k=1}^\infty (1-r^{2k}\zeta \alpha_0^{-1})(1-r^{2k}\zeta^{-1} \alpha_0)}{(1-\zeta\overline{\alpha_0})\prod_{k=1}^\infty (1-r^{2k}\zeta\overline{\alpha_0})(1-r^{2k}\zeta^{-1}(\overline{\alpha_0})^{-1})}\Big)
\end{align*}
is a single-valued holomorphic function in $\mathbb A_r$ (with respect to $z$) because the numerator and denominator of $h$ have the same number of zeros inside of $|z|=r$.

The new numerical scheme was verified against the exact solution with $\gamma=1/5, r=1/2$, and $\alpha=\Psi(3/4)$. The coefficients $\{a_n,b_n: n\geq 1\}$ decay quickly, so $N$ was chosen to be $N=16$ and $K=16$. The maximum difference between the approximated solution and exact solution was of the order $\mathcal O(10^{-3})$ where $100$ sample points were chosen evenly between $[0,2\pi]$. The condition number of the matrix arising from the system generated by the equations (\ref{vortexEqns}) was 3.9. Figure \ref{pointvortexGraph} shows $\text{Re} (h_D(\Psi(e^{i\theta}))$ and $\text{Re} (h_D(\Psi(re^{i\theta}))$ plotted against the exact solution.
\begin{figure}
\centering
\includegraphics[width=0.6\textwidth]{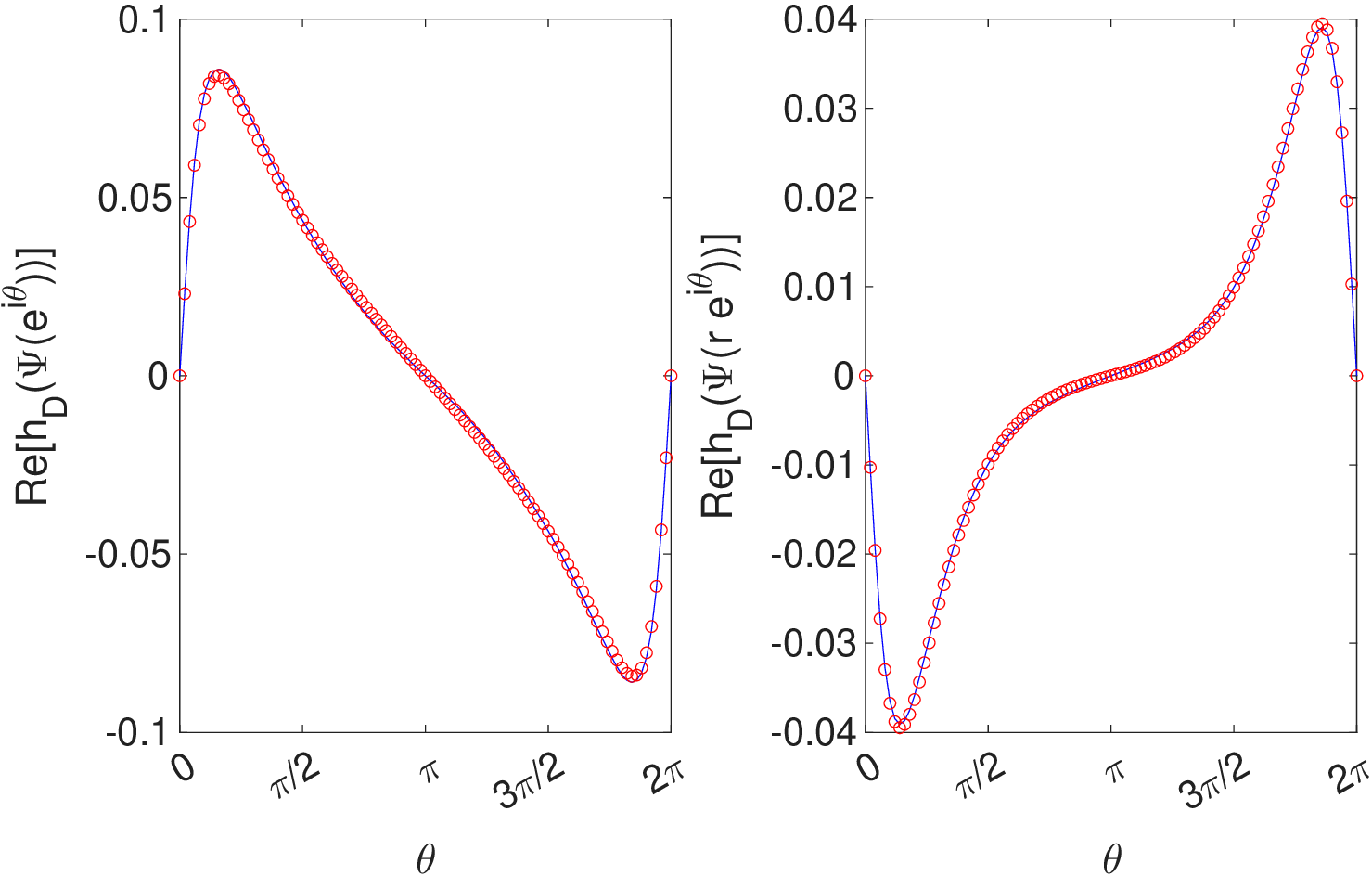}
\caption{The graphs of the approximated $\text{Re} (h_D(\Psi(e^{i\theta}))$ and $\text{Re} (h_D(\Psi(re^{i\theta}))$ plotted as a solid blue line and the exact solution plotted as red circles. }
\label{pointvortexGraph}
\end{figure}

\section{Appendix}
\subsection{Equality (\ref{specdecom1})}\label{appendixA}
We will prove the equality stated in equation (\ref{specdecom1}).
\begin{figure*}[t!]
\centering\includegraphics[width=0.5\textwidth]{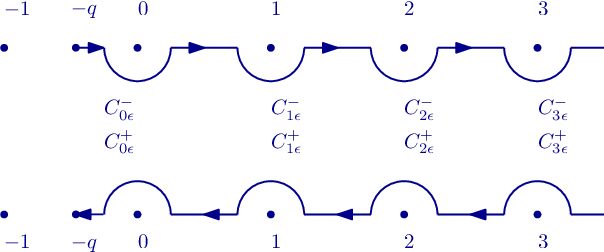}
        \caption{The contours $L_+$ (lower) and $L_-$ (higher)}
        \label{L+L-Contours}
    \end{figure*}
We use a very similar argument as was first presented in \cite{C15}. First, for $|z|<1$, set
\begin{align*}
    I=\int_{L_1}\frac{1}{1+r^{2n+1}}\frac{z^k}{1-e^{2\pi ik}}dk+\int_{L_2}\frac{z^k}{1+r^{2n+1}}dk+\int_{L_3}\frac{1}{1+r^{2n+1}}\frac{e^{2\pi ik}}{1-e^{2\pi ik}}z^kdk.
\end{align*}
Let $\mathcal P$ denote the principal value integral. The second integral in the equation above can be rewritten as
\begin{align*}
    &\int_{L_2}\frac{z^k}{1+r^{2n+1}}dk=\mathcal P\int_{L_2}\Big[\frac{1-e^{2\pi ik}}{1-e^{2\pi ik}}\Big]\frac{z^k}{1+r^{2n+1}}dk\\&=\mathcal P\int_{L_2}\Big[\frac{1}{1-e^{2\pi ik}}\Big]\frac{z^k}{1+r^{2n+1}}dk-\mathcal P\int_{L_2}\Big[\frac{e^{2\pi ik}}{1-e^{2\pi ik}}\Big]\frac{z^k}{1+r^{2n+1}}dk.
\end{align*}
Let $L_+$ and $L_2$ be the contours shown in figure (\ref{L+L-Contours}). The contour $L_+$ travels along the $x$-axis from right to left, has semi-circles of radius $\epsilon$ centered on the non-negative integers, and terminates at $-q$. The contour $L_-$ travels along the $x$-axis from left to right, has semi-circles of radius $\epsilon$ centered on the non-negative integers,
and starts at $-q$. Let the semi-circles of radius $\epsilon$ centered at $n$ on the contour $L_+$ be labeled $C_{n\epsilon}^+$, and let the semi-circles of radius $\epsilon$ centered at $n$ on the contour $L_-$ be labeled $C_{n\epsilon}^-$. Then

\begin{align*}
    \lim_{\epsilon\rightarrow 0}\int_{L_-}\Big[&\frac{1}{1-e^{2\pi ik}}\Big]\frac{z^k}{1+r^{2n+1}}dk\\&=+\mathcal P\int_{L_2}\Big[\frac{1}{1-e^{2\pi ik}}\Big]\frac{z^k}{1+r^{2n+1}}dk+\sum_{n=0}^\infty\lim_{\epsilon\rightarrow0}\int_{C_{n\epsilon}^-}\frac{z^k}{1+r^{2n+1}}\Big[\frac{1}{1-e^{2\pi ik}}\Big]dk,
\end{align*}
and 
\begin{align*}
    \lim_{\epsilon\rightarrow 0}\int_{L_+}\Big[&\frac{e^{2\pi ik}}{1-e^{2\pi ik}}\Big]\frac{z^k}{1+r^{2n+1}}dk\\&=+\mathcal P\int_{L_2}\Big[\frac{e^{2\pi ik}}{1-e^{2\pi ik}}\Big]\frac{z^k}{1+r^{2n+1}}dk+\sum_{n=0}^\infty\lim_{\epsilon\rightarrow0}\int_{C_{n\epsilon}^+}\frac{z^k}{1+r^{2n+1}}\Big[\frac{e^{2\pi ik}}{1-e^{2\pi ik}}\Big]dk.
\end{align*}
 Now, using the fractional residue theorem, we compute that
\begin{align*}
    &\lim_{\epsilon\rightarrow0}\int_{C_{n\epsilon}^-}\frac{z^k}{1+r^{2n+1}}\Big[\frac{1}{1-e^{2\pi ik}}\Big]dk=\pi i \text{ Res}\Big [\frac{z^k}{1+r^{2n+1}}\frac{1}{1-e^{2\pi ik}},n\Big]\\&=\pi i \frac{z^n}{1+r^{2n+1}}\lim_{k\rightarrow n}\frac{(k-n)}{1-e^{2\pi ik}}=\frac{1}{2} \frac{z^n}{1+r^{2n+1}}.
\end{align*}
Likewise
\begin{align*}
    &\lim_{\epsilon\rightarrow0}\int_{C_{n\epsilon}^+}\frac{z^k}{1+r^{2n+3}}\Big[\frac{e^{2\pi ik}}{1-e^{2\pi ik}}\Big]dk=\frac{1}{2} \frac{z^n}{1+r^{2n+3}}.
\end{align*}
Define the contours $\mathcal L_+$ and $\mathcal L_-$ as $L_- \equiv L_1 \cup L_-$ and $\mathcal L_+ \equiv L_+ \cup L_3$ pictured in figure (\ref{ScriptL_L+contours}).

\begin{figure} 
\centering\includegraphics[width=0.6\textwidth]{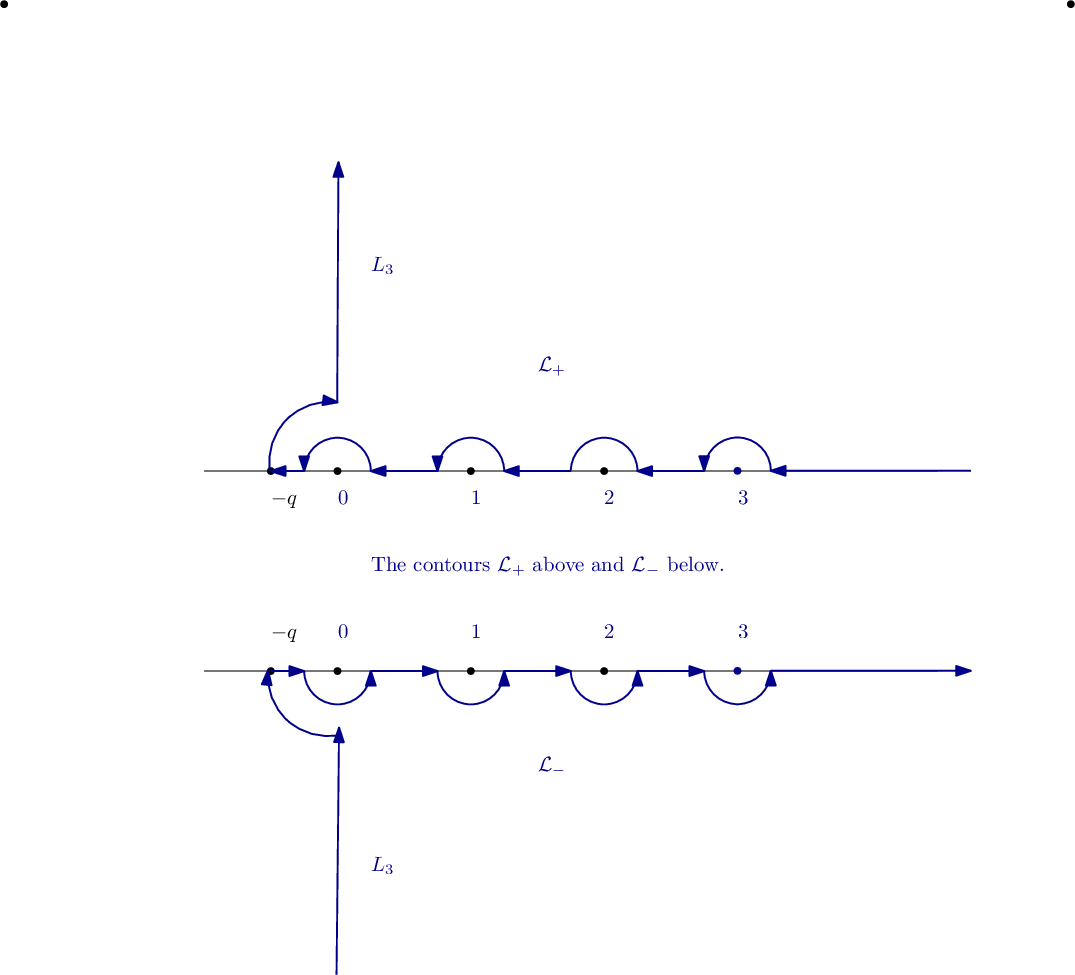}

    \caption{The contours $\mathcal L_+$ and $\mathcal L_-$ } \label{ScriptL_L+contours}
\end{figure}

\noindent Thus
\begin{align*}
    I=\int_{\mathcal L_-}\frac{1}{1+r^{2k+1}}\Big[ \frac{1}{1-e^{2\pi ik}}\Big]z^kdk+\sum_{n=0}^\infty  \frac{z^n}{1+r^{2n+1}}+\int_{\mathcal L_+}\frac{1}{1+r^{2k+1}}\Big[ \frac{e^{2\pi ik}}{1-e^{2\pi ik}}\Big]z^kdk.
\end{align*}
Note that $1+r^{2k+1}=0$ for $k\in\mathbb \C$ when 
\begin{align*}
    k=-\frac{1}{2}+i\frac{(2n+1)\pi}{2\ln(r)}, \hspace{1cm} n\in\mathbb Z.
\end{align*}
% Indeed,
% \begin{align*}
%     r^{2k+1}+1=0 \Longleftrightarrow r^{2k+1}=-1 \Longleftrightarrow e^{\log(r)(2k+1)}=e^{i\pi(1+2n)}.
% \end{align*}
% Thus
% \begin{align*}
%    \log(r)(2k+1)= i\pi(1+2n).
% \end{align*}
% Now solve for $k$:
% \begin{align*}
%      k=-\frac{1}{2}+i\frac{(2n+1)\pi}{2\ln(r)}, \hspace{1cm} n\in\mathbb Z.
% \end{align*}
Next, for $k=ix$ with $x\in \R$ we have
\begin{align*}
    \Big|\frac{1}{1+r^{2k+1}}\Big|=\Big|\frac{1}{1+r^{2ix}r}\Big|\leq\frac{1}{1-r}.
\end{align*}
Thus the integrals over $L_1$ and $L_3$ exponentially decay. Next, let $z=x+iy$ with $x>0$. Then
\begin{align*}
     \Big|\frac{1}{1+r^{2k+1}}\Big|= \Big|\frac{1}{1+r^{2x}r^{2yi}r}\Big|. 
\end{align*}
Given that $x>0$, it follows that 
\begin{align*}
    \Big|\frac{1}{1+r^{2x}r^{2yi}r}\Big|\rightarrow 1
\end{align*}
 as $|k|\rightarrow \infty$ when $z$ is in the first and fourth quadrant. Next, note that $\frac{1}{1+r^{2k+1}}\Big[ \frac{1}{1-e^{2\pi ik}}\Big]z^k$ is analytic in the fourth quadrant and $\frac{1}{1+r^{2k+1}}\Big[ \frac{e^{2\pi ik}}{1-e^{2\pi ik}}\Big]z^k$ is analytic in the first quadrant, so both integral vanish. 
Repeating this argument for 
\begin{align*}
    \frac{r}{z\overline \zeta}\sum_{n=0}^{\infty}\Big(\frac{r^2}{z\overline \zeta}\Big)^n\frac{1}{1+r^{2n+1}}
\end{align*}
gives (\ref{specdecoms2}).

\subsection{Equation (\ref{transformAnnulus})}\label{appendixB}
We prove equation (\ref{transformAnnulus}) under the conditions of Theorem 
(\ref{transformAnnulus}). The following computation will completed for 
% \begin{align*}
%     &f(z)=\int_{\partial \mathbb A_r}f(\zeta)\overline{S(\zeta,z)}d\sigma(\zeta) \\&=\frac{1}{2\pi}\Big[ \int_{\partial \mathbb A_r}\int_{L_1}f(\zeta) \frac{\gamma(z\overline{\zeta},k)}{1-e^{2\pi ik}}dkd\zeta+\int_{\partial \mathbb A_r}\int_{L_2}\gamma(z\overline{\zeta},k)dkd\zeta
%     +\int_{\partial \mathbb A_r}\int_{L_3}\frac{e^{2\pi ik}}{1-e^{2\pi ik}}\gamma(z\overline{\zeta},k)dkd\zeta\Big].
% \end{align*}

\begin{align*}
    \int_{\partial \mathbb A_r}\int_{L_2}f(\zeta)\gamma(z\overline{\zeta},k)dkd\zeta,
\end{align*}
and the other computations are similar. First note that
\begin{align*}
    &\int_{\partial \mathbb A_r}\int_{L_2}\gamma(z\overline{\zeta},k)dkd\sigma(\zeta)=\int_{L_2}\int_{\partial \mathbb A_r}f(\zeta)\frac{1}{1+r^{2k+1}}\Big((z\overline{\zeta})^k+\frac{r^{2k+1}}{(z\overline{\zeta})^{k+1}}\Big)d\sigma(\zeta) dk
    \\&=\int_{L_2}\frac{1}{1+r^{2k+1}}\Big[\int_{\partial\D}f(\zeta)\Big((z\overline{\zeta})^k+\frac{r^{2k+1}}{(z\overline{\zeta})^{k+1}}\Big)d\sigma(\zeta)+\int_{D_r(0)}f(\zeta)\Big((z\overline{\zeta})^k+\frac{r^{2k+1}}{(z\overline{\zeta})^{k+1}}\Big)d\sigma(\zeta)\Big].
\end{align*}
On $\partial \D$, we have $\zeta_1=\zeta_1(t)=e^{it}$. Thus $\overline{\zeta_1}=1/\zeta_1$ and 

\begin{align*}
    d\sigma(\zeta_1)=\overline{T_{\partial \D}(\zeta_1)}d\zeta_1=\overline{ie^{i t}}d\zeta_1=-i\overline{\zeta_1(t)}d\zeta_1=-i\overline\zeta_1 d\zeta_1=\frac{1}{i\zeta_1}d\zeta_1.
\end{align*}
Thus

\begin{align*}
    &\int_{\partial\D}f(\zeta)\Big((z\overline{\zeta})^k+\frac{r^{2k+1}}{(z\overline{\zeta})^{k+1}}\Big)d\sigma(\zeta)=\frac{1}{i}\int_{\partial\D}f(\zeta)\frac{z^k}{\zeta^{k+1}}d\zeta+\frac{r^{2k+1}}{i}\int_{\partial\D}\frac{f(\zeta)}{z^{k+1}}\zeta^k d\zeta
    \\&=\frac{1}{i}z^k\rho_{11}(k)+\frac{1}{iz^{k+1}}\rho_{21}(k).
\end{align*}
On $\partial D_r(0)$, we have $\zeta_2(t)=-re^{it}$. Thus
\begin{align*}
    d\sigma(\zeta_2)=\overline{T_{\partial \D}(\zeta_2)}d\zeta_2=\frac{\overline{-ire^{i t}}}{r}d\zeta_2(t)=ie^{-it}d\zeta_2(t)=-i\frac{\overline{\zeta_2(t)}}{r}d\zeta_2(t)=-i\frac{\overline\zeta_2}{r} d\zeta_2=\frac{r}{i\zeta_2}d\zeta_2.
\end{align*}
Also note that $\overline{\zeta_2}=\overline{-re^{it}}=r^2/(-re^{it})=r^2/\zeta_2$. Thus
\begin{align*}
    &\int_{D_r(0)}f(\zeta)\Big((z\overline{\zeta})^k+\frac{r^{2k+1}}{(z\overline{\zeta})^{k+1}}\Big)d\sigma(\zeta)=\frac{1}{i}\int_{D_r(0)}f(\zeta)z^k\frac{r^{2k}}{\zeta^k}\frac{r}{\zeta}d\zeta+\frac{1}{i}\int_{D_r(0)}f(\zeta)\frac{r^{2k+1}}{z^{k+1}}\frac{\zeta^{k+1}}{r^{2k+2}}\frac{r}{\zeta}d\zeta
    \\&=\frac{r^{2k+1}}{i}\int_{D_r(0)}f(\zeta)\frac{z^k}{\zeta^{k+1}}d\zeta+\frac{1}{i}\int_{D_r(0)}\frac{f(\zeta)}{z^{k+1}}\zeta^{k}d\zeta
    \\&=\frac{1}{i}z^k\rho_{22}(-k-1)+\frac{1}{iz^{k+1}}\rho_{12}(-k-1).
\end{align*}
Thus
\begin{align*}
    &\frac{1}{2\pi}\int_{\partial \mathbb A_r}\int_{L_2}f(\zeta)\gamma(z\overline{\zeta},k)dkd\zeta\\&=\frac{1}{2\pi i}\int_{L_2}\frac{1}{1+r^{2k+1}}\Big[z^k\rho_{11}(k)+\frac{1}{z^{k+1}}\rho_{21}(k)+z^k\rho_{22}(-k-1)+\frac{1}{z^{k+1}}\rho_{12}(-k-1)\Big]dk
    \\&=\frac{1}{2\pi i}\int_{L_2}\frac{1}{1+r^{2k+1}}\Big[z^k(\rho_{11}(k)+\rho_{22}(-k-1))+\frac{1}{z^{k+1}}(\rho_{21}(k)+\rho_{12}(-k-1))\Big]dk.
\end{align*}

\bibliography{bib}

@book{Bell:1992,
	address = {New York},
	author = {S. R. Bell},
	edition = {2nd},
	publisher = {CRC Press},
	title = {The Cauchy transform, potential theory, and conformal mapping},
	year = 2015}

@book{Conway:1990,
	address = {New York},
	author = {J. B. Conway},
	edition = {},
	publisher = {Springer-Verlag},
	title = {A course in functional analysis},
	year = 1990}

@article{FO97,
	author = {A.S. Fokas.},
	journal = {Proc. R. Soc. A},
	number = {},
    issue = {1962},
	pages = {1411–1443},
	title = {A Unified transform method for solving linear and certain Nonlinear {PDE}s.},
	volume = {453},
	year = {1997}}

@article{FK03,
	author = {A.S. Fokas and A.A. Kapaev.},
	journal = {IMA J. of Appl. Math},
	number = {},
    issue = {4},
	pages = {355-408},
	title = {On the transform method for the {L}aplace equation in a polygon.},
	volume = {68},
	year = {2003}}

@article{MC05,
	author = {J Marshal and D Crowdy},
	journal = {Phys. Fluids},
	number = {056602},
    issue = {5},
	pages = {1411–1443},
	title = {The motion of a point vortex around multiple circular islands},
	volume = {17},
	year = {2005}}

@article{JM04,
	author = {E.R. Johnson and N. R. McDonald},
	journal = {Proc. R. Soc. A},
	number = {},
    issue = {2044},
	pages = {939-954},
	title = {The motion of a vortex near two circular cylinders},
	volume = {460},
	year = {2004}}

@article{CL18,
	author = {E Luca and D. Growdy},
	journal = {IMA J. of Appl. Math},
	number = {},
    issue = {6},
	pages = {942-972},
	title = {A transform method for the biharmonic equation in multiply connected circular domains},
	volume = {83},
	year = {2018}}

@book{PD70,
	author = {Peter Duren},
    address = {New York},
	publisher = {Academic Press},
	number = {},
	pages = {},
	title = {Theory of $H^p$ Spaces},
	year = {1970}}

@article{SF10,
	author = {E. Spence and A. S. Fokas},
	journal = {Proc. R. Soc. A},
	pages = {2259--2281},
	title = {A new transform method {I}: domain-dependent fundamental solutions and integral representations},
	volume = {466},
	year = {2010}}

@article{DB14,
	author = {C Davis and B Fornberg},
	journal = {Complex Var. Elliptic Equ.},
	number = {},
    issue = {4},
	pages = { 564-577},
	title = {A spectrally accurate numberical implementation of the {F}okas transform method for {H}elmholtz-type {PDE}s},
	volume = {59},
	year = {2014}}

@article{DF15,
	author = {M Dimakos and AS Fokas},
	journal = {Stud. Appl. Math.},
	number = {},
    issue = {4},
	pages = {456-498},
	title = {The Poisson and the biharmonic equations in the interior of a convex polygon},
	volume = {134},
	year = {2015}}

@article{CH21,
	author = {J Hauge and D Crowdy},
	journal = {IMA J. Appl. Math},
	number = {},
    issue = {6},
	pages = {1287-1326},
	title = {A new approach to the complex {H}elmholtz equation with application to diffusion wave fileds, impedance spectroscopy and unsteady {S}tokes flow},
	volume = {86},
	year = {2021}}

@article{HLLL24,
	author = {J. J. Hulse and L. Lanzani and S. G. {Llewellyn Smith} and E. Luca},
	
	journal = {IMA J. of Appl. Math},
	pages = {574--597},
	title = {A transform-based technique for solving boundary value problesm on convex planar domains},
	volume = {89},
	year = {2024}}

@book{GM09,
	address = {},
	author = {G Gasper and M Rahman},
	edition = {2nd},
	publisher = {Cambridge University Press},
Pages = {},
	Series = {Encyclopedia of Mathematics and its Applications},
	Title = {Basic Hypergeometric Series},
	year = 2009}

@article{C15,
	author = {D Crowdy},
	journal = {Comput. Methods Funct.
            Th.},
	number = {},
	pages = { 655 – 687},
        issue = {4},
	title = {Fourier–{M}ellin transforms for circular domains.},
	volume = {15},
	year = {2015}}

@article{HLLS25,
	author = {J. Hulse, L. Lanzani. S. Llewellyn Smith, E. Luca.},
	journal = {Proc. R. Soc. A},
	number = {},
	pages = { },
        issue = {2319},
	title = {The Unified Transform Method: Beyond Circulr or Convex Domains},
	volume = {481},
	year = {2025}}

@article{DC09,
	author = {D Crowdy},
	journal = {Complex Var. Elliptic Equ.},
	number = {},
	pages = {221–236},
        issue = {},
	title = {The {S}chwarz problem in multiply connected domains and the {S}chottky–{K}lein prime
function},
	volume = {53},
	year = {2008}}

@article{C15B,
	author = {D Crowdy},
	journal = {IMA J. Appl. Math},
	number = {},
	pages = {1902-1931},
        issue = {},
	title = {A transform method for {L}aplace’s equation in multiply connected
circular domains},
	volume = {80},
	year = {2015}}
\bibliographystyle{plain}

\end{document}